\RequirePackage[2022-06-01]{latexrelease}
\documentclass[hidelinks,onefigHmmnum,onetabnum]{siamart220329}

\usepackage{amsfonts}
\usepackage{graphicx}
\usepackage{epstopdf}
\usepackage{algpseudocode}
\usepackage{enumitem}
\usepackage{ifoddpage}

\usepackage{dsfont}
\usepackage{subcaption}

\ifpdf
  \DeclareGraphicsExtensions{.eps,.pdf,.png,.jpg}
\else
  \DeclareGraphicsExtensions{.eps}
\fi

\newsiamremark{remark}{Remark}
\newsiamremark{hypothesis}{Hypothesis}
\crefname{hypothesis}{Hypothesis}{Hypotheses}
\newsiamthm{claim}{Claim}

\headers{$\mathcal{H}_2$ model reduction for BQO systems}{H. Fa\ss bender, S. Gugercin and T. Peters}

\title{
Bilinear Systems with Quadratic Outputs: \\ $\mathcal{H}_2$ Analysis, Optimality Conditions for Model Reduction, and Algorithmic Solutions 
\thanks{Submitted to the editors DATE.
\funding{Part of this work was completed during a research stay of the third author at Virginia Tech supported by a fellowship of the German Academic Exchange Service (DAAD). The work of Gugercin was supported in part by the National Science Foundation (NSF), United States under Grant No. CMMI-2130695}}}

\author{Heike Fa\ss bender \thanks{Institute for Numerical Analysis, TU Braunschweig
  (\email{h.fassbender@tu-braunschweig.de}).} 
\and Serkan Gugercin\thanks{Department of Mathematics and Division of Computational Modeling and Data Analytics, Academy of Data Science, Virginia Tech
  (\email{gugercin@vt.edu}).}\and Till Peters\thanks{Institute for Numerical Analysis, TU Braunschweig 
  (\email{till.peters@tu-braunschweig.de}), corresponding author.}}

\DeclareMathOperator{\diag}{diag}

\DeclareMathOperator{\tr}{tr}

\DeclareMathOperator{\vect}{vec}
\DeclareMathOperator{\orth}{orth}

\newcommand{\R}{\mathbb{R}}

\newcommand{\sH}{\mathcal{H}}
\newcommand{\norm}[1]{\left\lVert#1\right\rVert}
\newcommand{\abs}[1]{\left\lvert#1\right\rvert}
\newcommand{\skp}[2]{\left\langle {#1}, {#2} \right\rangle}

\ifpdf
\hypersetup{
  pdftitle={Bilinear quadratic output systems and related model reduction techniques},
  pdfauthor={H. Fa\ss bender, S. Gugercin and T. Peters}
}
\fi

\begin{document}

\maketitle

\begin{abstract}
    Bilinear systems with quadratic outputs (BQO) have recently 
emerged as an important system class, arising naturally in 
applications where both the dynamics and the quantities of interest 
depend nonlinearly on the state. Despite the growing interest in 
this class of systems, a systematic $\mathcal{H}_2$ framework for BQO systems 
has been lacking. In this paper, we develop such a framework by 
establishing an $\mathcal{H}_2$ inner product and norm for BQO 
systems, deriving output bounds in terms of 
the $\mathcal{H}_2$ norm, and obtaining first-order optimality 
conditions for $\mathcal{H}_2$ optimal model reduction. Building on 
these theoretical foundations, we propose an algorithm that computes 
a reduced BQO system satisfying these optimality conditions, and thus 
generalizing existing $\mathcal{H}_2$ optimal methods for bilinear 
and linear quadratic output systems. The effectiveness of the 
proposed framework is demonstrated on two numerical examples.
\end{abstract}

\begin{keywords}
    $\mathcal{H}_2$ model order reduction, balanced truncation, bilinear systems, linear quadratic output systems
\end{keywords}

\begin{MSCcodes}
93A15, 93B40, 93C10, 93C15
\end{MSCcodes}

\section{Introduction}
\label{sec:1}

Bilinear systems with quadratic outputs (BQO systems) arise naturally in applications where state-dependent nonlinearities appear in both the dynamics and the measured response. Following the framework introduced in~\cite{padhi_2024, FasGP25}, such a system takes the state-space form
\begin{subequations}\label{eq:BQO}
\begin{align}
  \dot{x}(t) &= Ax(t) + \sum_{k=1}^{m} N_k x(t)\, u_k(t) + Bu(t), 
              \quad x(0) = 0, \label{eq:BQO_state} \\
  y(t)       &= Cx(t)+\begin{bmatrix}
        x(t)^TM_1x(t) \\ \vdots \\ x(t)^TM_px(t)
    \end{bmatrix} = Cx(t) + \mathfrak{M}\bigl(x(t) \otimes x(t)\bigr),
              \label{eq:BQO_output}
\end{align}
\end{subequations}
where $ \mathfrak{M} 
    = \begin{bmatrix} \vect(M_1) \, \dots \, \vect(M_p) 
       \end{bmatrix}^T.$
Here, $A \in \mathbb{R}^{n \times n}$ is a stable matrix, i.e., all eigenvalues of $A$ lie in the left half-plane, 
$B \in \mathbb{R}^{n \times m}$, $C \in \mathbb{R}^{p \times n}$, 
$N_k \in \mathbb{R}^{n \times n}$ for $k = 1, \ldots, m$, and 
$M_j \in \mathbb{R}^{n \times n}$ for $j = 1, \ldots, p$.  Without loss of generality, we assume that each matrix $M_j$ is 
symmetric. We take the time horizon to be $t \in [0, \infty)$.

For notational convenience, we denote the BQO systems of the form~\eqref{eq:BQO} by
$\mathcal{S} = (A, B, C, \mathcal{N}, \mathcal{M})$, where the aggregated 
bilinear and quadratic-output matrix blocks are defined as
\begin{equation}\label{eq:NM_def}
  \mathcal{N} 
    := \begin{bmatrix} N_1 \, \dots \, N_m \end{bmatrix} 
       \in \mathbb{R}^{n \times nm},
  \qquad
  \mathcal{M} 
    := \begin{bmatrix} M_1 \, \dots \, M_p \end{bmatrix} 
       \in \mathbb{R}^{n \times np}.
\end{equation}

BQO systems unify and extend two well-studied system classes.  Setting $M_1 = \dots = M_p = 0$ recovers the classical bilinear  system (see, e.g., \cite{al-baiyat_new_1993,zhang_2002}), for which  the output reduces to the linear map $y(t) = Cx(t)$. Setting instead $N_1 = \dots = N_m = 0$ recovers the linear quadratic output (LQO) system (see, e.g., \cite{benner_lqo_2022,reiter_2024}), with dynamics $\dot{x}(t) = Ax(t) + Bu(t)$. BQO systems thus sit at the  intersection of bilinear dynamics and quadratic output maps, combining the challenges of both.

An important class of BQO systems arises when studying bilinear port-Hamilto\-nian systems of the form
\begin{subequations}\label{eq:BPH} 
\begin{align}
    \dot{x}(t) &= (J-R)Qx(t) + \left(B+Nx(t)\right)u(t),\\
    y(t) &= \left(B+Nx(t)\right)^TQx(t),
\end{align}
\end{subequations}
where $J,R,Q,N\in\mathbb{R}^{n \times n}$ and $B\in\mathbb{R}^{n\times 1}$ satisfy
$J=-J^T$, $R=R^T\geq 0$, and $Q=Q^T\geq 0$; see~\cite{MehU23,Van06}. The single-input single-output case ($m=p=1$) presented here can be easily extended to the multi-input multi-output setting for bilinear port-Hamiltonian systems. More generally, BQO systems can also arise as approximations of nonlinear systems, for instance, via Carleman bilinearization~\cite{rugh_nonlinear_1981}; see also \cite{hart13,mohler}.
Quadratic outputs as in \cref{eq:BQO_output} appear, for example, in modeling variances of stochastic systems
or energy-related quantities; see \cite{PulA19,Pul23,YueM13,ReiW24,VanVNLM12}.

In this work, we consider large-scale BQO systems, which typically arise, for example, from spatial discretizations of underlying partial differential equations.
This motivates the use of model order reduction to reduce the system dimension while preserving the input--output behavior up to a prescribed accuracy.
More precisely, we seek a reduced BQO system $\widehat{\mathcal{S}}=(\widehat{A},\widehat{B},\widehat{C},\widehat{\mathcal{N}},\widehat{\mathcal{M}})$ of order $r \ll n$ in state-space form
\begin{subequations}\label{eq:BQOr}
    \begin{align}
    \dot{\widehat{x}}(t) &= \widehat{A}\widehat{x}(t) + \sum_{k=1}^m \widehat{N}_k \widehat{x}(t) u_k(t)+ \widehat{B}u(t), \qquad \widehat{x}(0)=0, \label{eq:BQOr_a1}\\
    \widehat{y}(t)&=\widehat{C}\widehat{x}(t)+\begin{bmatrix}
        \widehat{x}(t)^T\widehat{M}_1\widehat{x}(t) \\ \vdots \\ \widehat{x}(t)^T\widehat{M}_p\widehat{x}(t)
    \end{bmatrix} = \widehat{C}\widehat{x}(t)+\widehat{\mathfrak{M}}\bigl(\widehat{x}(t)\otimes \widehat{x}(t)\bigr), \label{eq:BQOr_a2}
\end{align}
\end{subequations}
with $\widehat{A}\in\mathbb{R}^{r\times r}$, $\widehat{B}\in\mathbb{R}^{r \times m}$, $\widehat{C}\in \mathbb{R}^{p\times r}$,
$\widehat{N}_k\in\mathbb{R}^{r\times r}$ for $k=1,\dots, m$, $\widehat{M}_j\in \mathbb{R}^{r\times r}$ for $j=1,\dots, p$, and $ \widehat{\mathfrak{M}} 
    = \begin{bmatrix} \vect(\widehat{M}_1) \, \dots  \, \vect(\widehat{M}_p) 
       \end{bmatrix}^T$.
Similarly to the full BQO system, the block structures are defined as
\begin{equation*}
  \widehat{\mathcal{N}} 
    := \begin{bmatrix} \widehat{N}_1 \, \dots \, \widehat{N}_m \end{bmatrix} 
       \in \mathbb{R}^{r \times rm},
  \qquad
  \widehat{\mathcal{M}} 
    := \begin{bmatrix} \widehat{M}_1 \, \dots \, \widehat{M}_p \end{bmatrix} 
       \in \mathbb{R}^{r \times rp}.
\end{equation*}
The objective of model order reduction is to construct the reduced BQO system \cref{eq:BQOr} such that the reduced output $\widehat{y}(t)$ accurately approximates the original output $y(t)$ of the original BQO system~\eqref{eq:BQO}.

In this work, we develop an $\mathcal{H}_2$ model order reduction approach for BQO systems based on Petrov--Galerkin projections.
To this end, we approximate the state by $x(t) \approx V\widehat{x}(t)$, where $V \in \mathbb{R}^{n\times r}$ spans a low-dimensional subspace and  $\widehat{x}(t)$ is the reduced state, and choose a matrix $W \in \mathbb{R}^{n\times r}$ such that $W^TV$ is nonsingular.
Inserting this ansatz into \cref{eq:BQO} and projecting with $V(W^TV)^{-1}W^T$ yields the reduced BQO system \cref{eq:BQOr} with
\begin{align}
\begin{split}\label{eq_PetrovGalerkin}
& \widehat{A}=(W^TV)^{-1}W^TAV,\quad
\widehat{B}=(W^TV)^{-1}W^TB,\quad
\widehat{C}=CV,
\\
&\widehat{N}_k=(W^TV)^{-1}W^TN_kV, \quad k = 1, \ldots, m, \quad
\widehat{M}_j=V^TM_jV, \quad j = 1, \ldots, p.
\end{split}
\end{align}
The model reduction subspaces $V$ and $W$ completely determine the reduced model. Thus, the question becomes how to choose $V$ and $W$ for a high-fidelity approximant.

Balanced truncation (BT) for BQO systems has been investigated  in \cite{FasGP25} (see also \cite{padhi_2024,Gos24}), where $V$ and $W$ are chosen via the concept of balancing.  
The approach transforms the system into a balanced realization where reachability and observability Gramians coincide and are diagonal. The diagonal entries, known as Hankel singular values, measure the contribution of each state, and states associated with small singular values are truncated. To enable this construction, algebraic Gramians, generalizing the linear case and defined via generalized Lyapunov equations, are introduced and computed in \cite{FasGP25} (see also \cite{padhi_2024}).

In contrast, we propose an $\mathcal{H}_2$ based framework for 
choosing $V$ and $W$, developing what we believe is the first 
complete $\mathcal{H}_2$ theory for BQO systems. 
$\mathcal{H}_2$ 
model reduction has been extensively studied for bilinear systems 
(see, e.g., \cite{zhang_2002,benner_birka_2012,flagg_multipoint_2015,
redmann_bilinear_2021}) and for LQO systems 
(see, e.g., \cite{reiter_2024}), but no such framework has existed 
for BQO systems.
 The proposed framework naturally unifies and extends 
these two lines of work: the BQO setting recovers the existing 
$\mathcal{H}_2$ theories for bilinear and LQO systems as special 
cases. Building on the Gramians introduced in \cite{FasGP25}, we 
establish the $\mathcal{H}_2$ norm for BQO systems, derive optimality 
conditions for $\mathcal{H}_2$ optimal model reduction, and develop algorithms 
that construct reduced models satisfying these optimality conditions, 
thereby generalizing existing $\mathcal{H}_2$ optimal methods. 
Numerical experiments confirm that the proposed approach yields 
high-fidelity reduced models efficiently.
    
Our main contributions are the following.
\begin{itemize}
    \item We introduce an $\mathcal{H}_2$ inner product and norm for 
    BQO systems, extending the existing theory for bilinear and LQO 
    systems (\Cref{theo:normS}).
    \item We bound the output of a BQO system in terms of the 
    $\mathcal{H}_2$ norm of the system (\Cref{cor:yest}).
    \item We bound the output error between two BQO systems by the 
    $\mathcal{H}_2$ norm of their difference and a constant input 
    term (\Cref{cor:yerr_est}).
    \item We derive first-order optimality conditions for 
    $\mathcal{H}_2$ optimal model reduction of BQO systems 
    (\Cref{thm:grad}).
    \item We propose an algorithm that computes a reduced BQO system 
    satisfying these optimality conditions (\Cref{alg:BQO_TSIA}).
\end{itemize}

The paper is structured as follows: Preliminaries are discussed in \Cref{sec:2}. In \Cref{sec:3}, we review the definitions for Gramians of BQO systems and their associated Lyapunov equations. This section also includes definitions of similar generalized Sylvester equations. In \Cref{sec:4}, we develop the $\sH_2$ inner product and norm for BQO systems based on kernels for BQO systems. This new analysis is used in \Cref{sec:5} to estimate the output of a BQO system which leads to an estimate of the output error comparing two BQO systems. In \Cref{sec:6}, this error estimate motivates us to perform $\sH_2$ model order reduction for BQO systems where we derive optimality conditions and construct a iterative algorithm to find a suitable reduced BQO system. Finally, in \Cref{sec:7}, we present numerical results and compare the proposed algorithm with balanced truncation for BQO systems.

\section{Preliminaries}\label{sec:2}

In the following, we will make use of the standard Householder notation, labeling matrices with upper-case Roman letters ($A, B,$ etc.), vectors with lower-case Roman letters ($a, b,$ etc.), and scalars with lower-case Greek letters ($\alpha, \beta$, etc.). 
When labeling the columns or rows of a matrix, the columns are usually denoted by the corresponding lowercase Roman letter, indexed by an integer $i, j, k, m$, or $p$.

The spectral norm will be denoted by $\norm{\cdot}=\norm{\cdot}_2$. 
Furthermore, we use $X \geq 0$ to denote the symmetric positive semidefiniteness of a matrix $X \in \mathbb{R}^{n \times n}$. Similarly, $X > 0$ denotes a symmetric positive definite matrix $X$.

In addition, we will employ the standard Kronecker product notation as defined in, e.g., \cite[Sections 4.2, 4.3]{HorJ94} or \cite[pp.~27-28]{GolubVanLoan4th}. More precisely, we will use the fact that $(A\otimes B)(C \otimes D) = AC \otimes BD$ holds for matrices of appropriate size, and
\begin{subequations}  
\begin{align}  
\tr(X^TY)&= \vect(X)^T\vect(Y), \label{eq:id_tr} \\
\vect(XYZ)&=(Z^T\otimes X)\vect(Y),\label{eq:id_vect_kron}
\end{align} 
\end{subequations}
where  $\tr(A)$ denotes the trace of the square matrix $A$ and $\vect(X)$ the vectorization of a rectangular matrix $X$. We use $\diag(A_1,A_2)=\left[\begin{smallmatrix}
    A_1 & \\ &A_2
\end{smallmatrix}\right]$ to denote the block-diagonal matrix formed  by $A_1$ and $A_2$.
We note that these properties have already been used to re-write \eqref{eq:BQO_output} in a compact form with $\mathfrak{M}$ as done in the LQO case in~\cite{reiter_2024}. Moreover, let us recall that $\norm{I \otimes X} = \norm{X}$ holds. Whenever useful, we denote an identity matrix of dimension $n$ by $I_n$ instead of just $I$ in order to make to derivations easier to follow.

We will make use of the following fact:
\begin{proposition}{\cite[Lemma 2.1]{FasGP25}} \label{lem:lem1}
    Let $R\in\R^{n\times n}$ and 
    $\mathcal{G}:=\begin{bmatrix}
        G_1 \dots G_q 
    \end{bmatrix}, \mathcal{H}:=\begin{bmatrix}
        H_1^T  \dots  H_q^T 
    \end{bmatrix}^T$,
    with $G_k,H_k\in\R^{n\times n}$ for all $k=1,\dots, q$. Then, 
    \begin{align*}
        \mathcal{G}(I_q\otimes R)\mathcal{H}=\sum_{k=1}^qG_kRH_k.
    \end{align*}
\end{proposition}
This result immediately leads to the following corollary.
\begin{corollary}
    \label{lem:id_M}
    For $\mathfrak{M}=[\vect(M_1)  \, \dots \, \vect(M_p)]^T\in\R^{p\times n^2}$ and \\ $\widehat{\mathfrak{M}}=[\vect(\widehat{M}_1) \, \dots \, \vect(\widehat{M}_p)]^T\in\R^{p\times r^2}$
    with $X\in\R^{n\times r}$, it holds
    \begin{align*}
        \tr\left(\mathfrak{M}(X\otimes X)\widehat{\mathfrak{M}}^T\right)         = \sum_{j=1}^p\tr\left(X^TM_jX\widehat{M}_j\right).
    \end{align*}
\end{corollary}

In addition to the aggregated bilinear and quadratic-output matrix blocks $\mathcal{N}$ and $\mathcal{M}$ \eqref{eq:NM_def}, we will define the matrix
    \begin{align}
    \mathcal{N}_T&:=\begin{bmatrix}
        N_1^T \, \dots \, N_m^T
    \end{bmatrix}\in \R^{n\times nm}.\label{eq:mathcalNT}
\end{align}

 For computing gradients of error expressions with traces (as we will encounter later in the manuscript), we recall a lemma from \cite{mlinaric_structure-preserving} which enables us to compute these derivatives.
\begin{proposition}{\cite[Lemma 2.43]{mlinaric_structure-preserving}}\label{lem:gradtr}
    Let $A\in\R^{n\times m}, B\in \R^{n\times n}, C \in \R^{m\times m}$. Then $\nabla_{X}\tr(AX^T)=A$ and $\nabla_{X}\tr(BXCX^T)=BXC+B^TXC^T$ for $X\in\R^{n\times m}$.    
\end{proposition}

\section{Gramians and generalized Sylvester and Lyapunov equations} \label{sec:3}
The reachability and observability Gramians are the main tools behind balanced truncation based model reduction~\cite{BenB17,GugA04}.
Suitable choices for the reachability Gramian $P$ and the observability Gramian $Q$ associated with the BQO system~\cref{eq:BQO} have been considered in \cite{FasGP25}. The reachability Gramian coincides with the one for bilinear systems; see  \cite[Equation (8) for $H=0$]{benner_balanced_2024} (see also
\cite[Definition 3.3.1]{padhi_2024}). 
It is defined via the reachability mapping 
\[
\bar{P} = \begin{bmatrix} \bar{P}_1 & \bar{P}_2 & \dots \end{bmatrix}
\]
with
\begin{align}    \label{eq:Pbar}
\begin{split}
    \bar{P}_1(t_1)&=e^{At_1}B \in \R^{n\times m}, \\
    \bar{P}_i(t_1,\dots , t_i) &= e^{At_i}\mathcal{N}(I_m\otimes \bar{P}_{i-1}(t_1,\dots,t_{i-1})) \in\R^{n\times m^i}, \quad \text{for } i\geq 2.
    \end{split}
\end{align}
With this, the reachability Gramian is defined as $P=\sum_{i=1}^\infty P_i \in \R^{n\times n}$ where
\begin{align*}
    P_i=\int_0^\infty \dots \int_0^\infty \bar{P}_i(t_1,\dots,t_i)\bar{P}^T_i(t_1,\dots,t_i)dt_1\dots dt_i,
\end{align*}
assuming all integrals exist and the infinite series converges.  
A similar derivation leads to the definition of  the observability Gramian $Q$; see \cite[Section 3.2.1]{FasGP25} where $Q$ is denoted by $Q^S$. We omit all details here, as they are not needed in the following. By construction, $P$ and $Q$ are symmetric positive semidefinite.

Recall that $A$ is assumed to be stable. In other words, there exist $\beta>0$ and $0<\alpha \leq -\max_i(\text{Re}(\lambda_i(A)))$ such that $\norm{e^{At}}\leq \beta e^{-\alpha t}$ for $t\geq 0$ \cite{adrianova_1995,willems70}. As discussed in \cite{FasGP25}, the parameters $\alpha$ and $\beta$ determine whether $P$ and $Q$ exist.

\begin{theorem}{\cite[Theorem 4.1, Theorem 4.4]{FasGP25}} 
\label{thm:PQ}
Consider a BQO system \cref{eq:BQO} with a stable matrix $A$.
    If $ \Gamma := \max\left\{\norm{\mathcal{N}}^2, \norm{\mathcal{N}_T}^2\right\} < \frac{2\alpha}{\beta^2}$, the reachability Gramian $P$ exists and the matrices $P$ and $Q$ are unique symmetric positive semidefinite solutions to 
\begin{subequations}\label{eq:PQ_full}
\begin{align}
    AP+PA^T+\sum_{k=1}^m N_kPN_k^T+BB^T=0,\label{eq:Pfull}\\
    A^TQ+QA+\sum_{k=1}^m N_k^TQN_k+\sum_{j=1}^pM_jPM_j +C^TC=0.\label{eq:Qfull}
\end{align}
\end{subequations}
\end{theorem}

\begin{remark} \label{rem:Qgramian} 
Given the stability of $A$, the condition $\norm{\mathcal{N}}^2 < \frac{2\alpha}{\beta^2}$ guarantees the existence of a unique solution to~\cref{eq:Pfull}; similarly, $\norm{\mathcal{N}_T}^2 < \frac{2\alpha}{\beta^2}$ ensures a unique solution to~\cref{eq:Qfull}. 
The stronger requirement $\Gamma_Q:=\max\left\{\norm{\mathcal{N}}^2,\norm{\mathcal{N}_T}^2,\norm{\mathcal{M}}^2\right\} < \frac{2\alpha}{\beta^2}$ is only needed for the definition of $Q$; see \cite[Remark~4.6]{FasGP25} for details.
Since \eqref{eq:Qfull} admits a solution under the weaker condition $\norm{\mathcal{N}_T}^2 < \frac{2\alpha}{\beta^2}$, such a solution does not need to be a Gramian. Accordingly, since in the sequel we focus on solutions of the Lyapunov equations~\eqref{eq:PQ_full}, we drop the assumption $\norm{\mathcal{M}}^2 < \frac{2\alpha}{\beta^2}$. Consequently, the solution of \eqref{eq:Qfull} is not necessarily a Gramian, and we therefore refrain from referring to $Q$ as a Gramian.
\end{remark}

If $\widehat{A}$ in~\eqref{eq:BQOr} is stable, Gramians can also be defined for the reduced system. In analogy to the preceding discussion, there exist constants $\widehat{\beta} > 0$ and $0 < \widehat{\alpha} \leq -\max_i \bigl(\mathrm{Re}(\lambda_i(\widehat{A}))\bigr)$ such that
$\|e^{\widehat{A}t}\| \leq \widehat{\beta} e^{-\widehat{\alpha} t},  t \geq 0.$

\begin{corollary}{\cite[Theorem 4.1, Theorem 4.4]{FasGP25}}
\label{thm:PQhat}
    Consider a BQO system \cref{eq:BQOr} with a stable matrix $\widehat{A}$. If $\widehat \Gamma := \max\left\{\norm{\widehat{\mathcal{N}}}^2, \norm{\widehat{\mathcal{N}}_T}^2\right\} < \frac{2\widehat{\alpha}}{\widehat{\beta}^2}$, the reachability Gramian $\widehat{P}$ exists and the matrices $\widehat{P}$ and $\widehat{Q}$ are unique symmetric positive semidefinite solutions to 
    \begin{subequations}\label{eq:PQ_red}
\begin{align}
    \widehat{A}\widehat{P} + \widehat{P}\widehat{A}^T + \sum_{k=1}^m \widehat{N}_k \widehat{P} \widehat{N}_k^T + \widehat{B}\widehat{B}^T &= 0, \label{eq:P_red}\\
    \widehat{A}^T\widehat{Q} + \widehat{Q}\widehat{A} + \sum_{k=1}^m \widehat{N}_k^T \widehat{Q} \widehat{N}_k + \sum_{j=1}^p \widehat{M}_j \widehat{P} \widehat{M}_j + \widehat{C}^T \widehat{C} &= 0. \label{eq:Q_red}
\end{align}
\end{subequations}
\end{corollary}
Note that $\widehat{P}$, as the solution to~\cref{eq:P_red},  is the reachability Gramian the reduced system $\widehat{\mathcal{S}}$, while $\widehat{Q}$, as the solution to~\cref{eq:Q_red}, can only be interpreted as the observability Gramian of  $\widehat{\mathcal{S}}$ if $\norm{\widehat{\mathcal{M}}}^2 < \frac{2\widehat{\alpha}}{\widehat{\beta}^2}$ holds; see~\Cref{rem:Qgramian}.

As in \cite[Section 5]{benner_birka_2012}, we will make use of  the solution of certain generalized Sylvester equations in our model reduction algorithm. Using the reachability mapping $\bar{P}$ in \eqref{eq:Pbar} and its reduced equivalent $\widehat{\bar{P}}$, define $X\in\R^{n\times r}$ as
\begin{align} \label{eq:sumxi}
    X:=\sum_{i=1}^\infty X_i ,
\quad \text{using} \quad 
    X_i:=\int_0^\infty \dots \int_0^\infty \bar{P}_i(t_1,\dots,t_i)\widehat{\bar{P}}^T_i(t_1,\dots,t_i)dt_1\dots dt_i
\end{align}
assuming all integrals exist and the infinite series converges. 
$X$ represents a connection between the full and the reduced system. Under standard assumptions, $X$ is the solution of a generalized Sylvester equation similar to \Cref{thm:PQ,thm:PQhat}.
\begin{theorem}\label{thm:SylvEq} 
    Consider the BQO systems \cref{eq:BQO} and \cref{eq:BQOr} with stable $A$ and $\widehat{A}$. Let $\alpha,\beta, \widehat{\alpha}, \widehat{\beta}$ be the associated stability parameters. If $\norm{\mathcal{N}}^2<2\alpha/\beta^2$ and $\norm{\widehat{\mathcal{N}}}^2<2\widehat{\alpha}/\widehat{\beta}^2$, then $X$ from \cref{eq:sumxi} exists and satisfies
     \begin{align}
        AX+X\widehat{A}^T+\sum_{k=1}^mN_kX\widehat{N}_k^T+B\widehat{B}^T=0.   \label{eq:thmXex}     
    \end{align}
    Furthermore, with $E,F\in\R^{n\times r}$, the generalized Sylvester equations
    \begin{subequations}
    \begin{align}
        AX+X\widehat{A}^T+\sum_{k=1}^mN_kX\widehat{N}_k^T+E=0, \label{eq:sylvE} \\  
        A^TY+Y\widehat{A}+\sum_{k=1}^mN_k^TY\widehat{N}_k+F=0 \label{eq:sylvF}
    \end{align}
    \end{subequations}
     have unique solutions $X\in\R^{n\times r}$ and $Y\in\R^{n\times r}$ if $\norm{\mathcal{N}}^2<2\alpha/\beta^2$ and $\norm{\widehat{\mathcal{N}}}^2<2\widehat{\alpha}/\widehat{\beta}^2$ for $X$, and if $\norm{\mathcal{N}_T}^2<2\alpha/\beta^2$ and $\norm{\widehat{\mathcal{N}}_T}^2<2\widehat{\alpha}/\widehat{\beta}^2$ for $\mathcal{N}_T$ as in \eqref{eq:mathcalNT} for $Y$, respectively.
\end{theorem}
\begin{proof} 
    Analogously to \cite[Theorem 4.1]{FasGP25}, one can show the existence of $X_i$ and $X$ under the given conditions. As in \cite[Section 3]{FasGP25} it can be shown that the $X_i$ satisfy Sylvester equations 
    \begin{align}
        AX_1+X_1\widehat{A}^T+B\widehat{B}^T= \, 0 \quad \text{and} \quad     
        AX_i+X_i\widehat{A}^T+\sum_{k=1}^mN_kX_{i-1}\widehat{N}_k^T=\, 0 \label{eq:eqXi}
    \end{align}
    for $i>1$. Summing over these equations leads to \cref{eq:thmXex}.       

Now we inspect the unique solution to \cref{eq:sylvE}. Similarly to \cite[Theorem 4.4]{FasGP25}, we can rewrite the generalized Sylvester equation \eqref{eq:sylvE} as a fixed-point equation $X=-\mathcal{L}_{A,\widehat{A}}^{-1}\left(\Pi_{N,\widehat{N}}(X)+E\right)$ using the operators $\mathcal{L}_{A,\widehat{A}}(X):=AX+X\widehat{A}^T$ and $\Pi_{N,\widehat{N}}(X):=\sum_{k=1}^mN_k X \widehat{N}_k^T$. This fixed-point equation has a unique solution if $\norm{\mathcal{L}_{A,\widehat{A}}^{-1}\Pi_{N,\widehat{N}}}<1$, which now follows due to the assumptions $\norm{\mathcal{N}}^2<2\alpha/\beta^2$ and $\norm{\widehat{\mathcal{N}}}^2<2\widehat{\alpha}/\widehat{\beta}^2$ (see also \cite[Lemma 5.1]{benner_birka_2012}).
The statement about the uniqueness of the solution to \cref{eq:sylvF} follows analogously.
\end{proof}

As mentioned in the proof above, \cref{eq:thmXex} can be solved via a sequence of fixed-point iterations. Let $\widetilde{X}_\ell:=\sum_{i=1}^\ell X_i$ denote the sum of the first $\ell$ terms $X_i$ in~\eqref{eq:sumxi}. Then $\widetilde X_1 = X_1$ holds with $A\tilde{X}_1+\tilde{X}_1\widehat{A}^T+B\widehat{B}^T= 0$ and summing over \cref{eq:eqXi} for $i=2,3,\dots,\ell$, yields
    \begin{align}
        A\widetilde{X}_\ell+\widetilde{X}_\ell\widehat{A}^T+\sum_{k=1}^mN_k\widetilde{X}_{\ell-1}\widehat{N}_k^T+B\widehat{B}^T=0. \label{eq:eqXitilde}
    \end{align}
 With the given conditions, these fixed-point iterations converge to $X$ (see \cite[Lemma 5.1]{benner_birka_2012} or for the case $r=n$, e.g., \cite{FasGP25} or \cite[Theorem 4]{benner_balanced_2024}). Hence, $X=\sum_{i=1}^\infty X_i=\lim_{\ell \to \infty}\widetilde{X}_\ell$ holds. This will be important for the computation of $X$ later.

\begin{remark}
    In \cite[Remark 4.2]{FasGP25}, it is noted that $\norm{\mathcal{N}}^2\leq \sum_{k=1}^m \norm{N_k}^2$ and $\norm{\mathcal{N}_T}^2\leq \sum_{k=1}^m \norm{N_k}^2$ holds. Similarly to \cite[Remark 4.2]{FasGP25}, one can substitute the requirements $\norm{\mathcal{N}}^2,\norm{\mathcal{N}_T}^2<2\alpha/\beta^2$ by $\sum_{k=1}^m \norm{N_k}^2<2\alpha/\beta^2$, which is a looser bound, but is used more frequently in the literature.
\end{remark}

With \Cref{thm:SylvEq}, we can now assume unique solutions $X\in\R^{n\times r}$, $Y\in\R^{n\times r}$ to
\begin{subequations}\label{eq:PQ_mix}
\begin{align}
    AX+X\widehat{A}^T+\sum_{k=1}^m N_kX\widehat{N}_k^T+B\widehat{B}^T=0, \label{eq:P_mix}\\
    A^TY+Y\widehat{A}+\sum_{k=1}^m N_k^TY\widehat{N}_k-\sum_{j=1}^pM_jX\widehat{M}_j -C^T\widehat{C}=0 \label{eq:Q_mix}
\end{align}
\end{subequations}
if $\norm{\mathcal{N}}^2,\norm{\mathcal{N}_T}^2<2\alpha/\beta^2$ and $\norm{\widehat{\mathcal{N}}}^2,\norm{\widehat{\mathcal{N}}_T}^2<2\widehat{\alpha}/\widehat{\beta}^2$.

This allows us to write state-space a description of the error system $\mathcal{S}_e=\mathcal{S}-\widehat{\mathcal{S}}$ between the BQO systems $\mathcal{S}$ and $\widehat{\mathcal{S}}$.
Define $M_{e,j}=\diag(M_j,-\widehat{M}_j)$ and 
let $x_e=\left[\begin{smallmatrix}
    x \\ \widehat{x}
\end{smallmatrix}\right]\in \mathbb{R}^{n+r}$ be the state of the error system $\mathcal{S}_e$. Then, $\mathcal{S}_e$ has the state-space form
\begin{align*}
    \dot{x}_e &= \underbrace{\diag(A,\widehat{A})}_{=A_e} x_e + \sum_{k=1}^m\underbrace{\diag(N_k,\widehat{N}_k)}_{=N_{e,k}} x_e u_k + \underbrace{\begin{bmatrix}B\\ \widehat{B}\end{bmatrix}}_{=B_e} u,\\
    y_e=y-\widehat{y}&=\underbrace{\begin{bmatrix}
        C & -\widehat{C}
    \end{bmatrix}}_{=C_e} x_e
    +\begin{bmatrix}
        x_e^T M_{e,1} x_e \\ 
    \vdots \\
    x_e^T M_{e,p} x_e 
    \end{bmatrix}.
    \end{align*}
Now, we can deduce a corollary from \cref{thm:PQ,thm:PQhat,thm:SylvEq} that states the generalized Lyapunov equations for the error system $\mathcal{S}_e$.
\begin{corollary} \label{cor:err_Lyap} Consider the BQO systems \cref{eq:BQO,eq:BQOr} with stable $A$ and $\widehat{A}$ with stability parameters $\alpha$, $\beta$ and $\widehat{\alpha}$, $\widehat{\beta}$, respectively. 
Furthermore, let $\mathcal{N}\in\R^{n\times nm}$ and $\widehat{\mathcal{N}}\in \R^{r\times rm}$ be as in \cref{eq:NM_def} and $\mathcal{N}_T\in\R^{n\times nm}$ and $\widehat{\mathcal{N}}_T\in \R^{r\times rm}$ be as in \cref{eq:mathcalNT}.
    If $\max\left\{\norm{\mathcal{N}}^2,\norm{\mathcal{N}_T}^2\right\}<2\alpha/\beta^2$ and $\max\left\{\norm{\widehat{\mathcal{N}}}^2,\norm{\widehat{\mathcal{N}}_T}^2\right\} <2\widehat{\alpha}/\widehat{\beta}^2$, then it holds
\begin{align*}
    A_eP_e+P_eA_e^T+\sum_{k=1}^mN_{e,k}P_eN_{e,k}^T+ B_eB_e^T &=0,\\
    A_e^TQ_e+Q_eA_e+\sum_{k=1}^mN_{e,k}^TQ_eN_{e,k} + \sum_{j=1}M_{e,j}P_eM_{e,j} + C_e^TC_e &=0
\end{align*}
with unique 
$P_e=\begin{bmatrix}
    P & X\\ X^T &\widehat{P}
\end{bmatrix}$ and $Q_e=\begin{bmatrix}
    Q & Y\\ Y^T &\widehat{Q}
\end{bmatrix}$ and unique solutions $P,Q,\widehat{P},\widehat{Q},X$, $Y$ to \cref{eq:Pfull,eq:P_mix,eq:P_red,eq:Qfull,eq:Q_mix,eq:Q_red}.
\end{corollary}

In the following, we will frequently use the following lemma which is a generalization of \cite[Lemma A.1]{yan_approximate_1999}.
\begin{lemma} \label{lem:sylv_tr} 
    Let $A\in\R^{n\times n}$, $B\in\R^{r\times r}$, $E\in\R^{n\times r}$, $F\in\R^{n\times r}$, and $G_k\in\R^{n\times n}$, $H_k\in\R^{r\times r}$, for all $k=1,\dots, m$. Assume that the generalized Sylvester equations    
        \begin{align}\label{eq:XY}
    AX+XB+\sum_{k=1}^mG_kXH_k^T+E=0
    ~\mbox{and}~
A^TY+YB^T+\sum_{k=1}^mG_k^TYH_k+F=0
    \end{align}
    have unique solutions $X\in\R^{n\times r}$ and $Y\in \R^{n\times r}$. Then it holds
    \begin{align*}
        \tr(E^TY)=\tr(F^TX).
    \end{align*}
\end{lemma}
\begin{proof}
    Vectorize the two Sylvester equations in \eqref{eq:XY} to obtain
    \begin{align*}
        \underbrace{\left((I\otimes A)+(B^T\otimes I)+\sum_{k=1}^m(H_k\otimes G_k)\right)}_{=:-\mathcal{G}} \vect(X)&=-\vect(E), \\
        \underbrace{\left((I\otimes A^T)+(B\otimes I)+\sum_{k=1}^m(H_k^T\otimes G_k^T)\right)}_{=-\mathcal{G}^T} \vect(Y)&=-\vect(F),  
    \end{align*}
    using \cref{eq:id_vect_kron}.
    Hence, it holds that $\mathcal{G}\vect(X)=\vect(E)$ and $\mathcal{G}^T\vect(Y)=\vect(F)$.
    Then, using \cref{eq:id_tr}, we obtain the desired statement
    \begin{align*}
        \tr(E^TY)&=\vect(E)^T\vect(Y) = \vect(E)^T\mathcal{G}^{-T}\vect(F)= \vect(X)^T\vect(F)=\tr(X^TF).
    \end{align*}
\end{proof}

\section{$\sH_2$ analysis for BQO systems}
\label{sec:H2analysys}
In this section, we will establish the $\sH_2$ theory for BQO systems. We will start by developing the $\sH_2$ inner product and $\sH_2$ norm in \Cref{sec:4} and then use this analysis to bound the norm of the output via the $\sH_2$ norm in
\Cref{sec:5}. 
\subsection{$\sH_2$ inner product and $\sH_2$ norm for BQO systems}
\label{sec:4}
The $\sH_2$ system norms have been defined for linear systems, but also for bilinear systems or linear systems with quadratic output (see, e.g., \cite{zhang_2002,benner_birka_2012,flagg_multipoint_2015,redmann_bilinear_2021,reiter_2024}). In each of these settings,  the $\sH_2$ norm has proven essential for effective model order reduction. Motivated by this, we develop the $\sH_2$ analysis and the $\sH_2$ norms here for BQO systems, laying the theoretical 
foundation for $\mathcal{H}_2$ optimal model reduction in the 
sections that follow.

Note that the solution $x(t)$ of \eqref{eq:BQO_state} can be written as $x(t)=\sum_{i=1}^\infty x_i(t)$ with
 \[
  x_i(t)    
     = \int_0^t\int_0^{t-t_i}\dots \int_0^{t-t_i-\dots- t_{2}} \bar{P}_i(t_1,\dots,t_i)\left( u^{(t_i)}(t) \otimes \dots \otimes u^{(t_i,\dots,t_1)}(t)\right) dt_1\dots dt_i
\]
with $u_k^{(t_1,\dots,t_l)}(t):=u_k(t-t_1-\dots -t_l)$.
With this, the output equation \cref{eq:BQO_output} of the BQO system, can be rewritten as
 \begin{align*}
     y(t)&= Cx(t) +\mathfrak{M}(x(t)\otimes x(t))
     = \sum_{i=1}^\infty Cx_i(t)+\sum_{i,j=1}^\infty \mathfrak{M}(x_i(t)\otimes x_j(t)) \\
     &= \sum_{i=1}^\infty \int\dots \int C\bar{P}_i \left( u\otimes \dots \otimes u\right) dt_1\dots dt_i \\
     &\qquad + \sum_{i,j=1}^\infty \int \dots \int \mathfrak{M} \left(\bar{P}_i\otimes \bar{P}_j\right) (u\otimes \dots \otimes u) dt_1\dots dt_id\bar{t}_1 \dots d\bar{t}_j,
 \end{align*}
 where we omit the $t_i$ from superscripts and arguments for simplicity.
 This inspires us to define the Volterra kernels for the full system $\mathcal{S}$ \cref{eq:BQO} and also for the reduced system $\widehat{\mathcal{S}}$ \cref{eq:BQOr}:
 \begin{subequations}     \label{eq:volt_kern}
 \begin{align}
     h_{C,i}(t_1,\dots, t_i)&:=C\bar{P}_i(t_1,\dots, t_i)\in\R^{p\times m^i}, \\
     h_{M,i,j}(t_1,\dots,t_i,\bar{t}_1,\dots \bar{t}_j) &:= \mathfrak{M}(\bar{P}_i(t_1,\dots, t_i)\otimes \bar{P}_j(\bar{t}_1,\dots, \bar{t}_j)) \in\R^{p\times m^{i+j}}, \\
     \widehat{h}_{\widehat{C},i}(t_1,\dots, t_i)&:=\widehat{C}\widehat{\bar{P}}_i(t_1,\dots, t_i)\in\R^{p\times m^i}, \\
     \widehat{h}_{\widehat{M},i,j}(t_1,\dots,t_i,\bar{t}_1,\dots \bar{t}_j) &:= \widehat{\mathfrak{M}}(\widehat{\bar{P}}_i(t_1,\dots, t_i)\otimes \widehat{\bar{P}}_j(\bar{t}_1,\dots, \bar{t}_j)) \in\R^{p\times m^{i+j}}.
 \end{align}
 \end{subequations}
 These kernels should represent the characterization of the input-output-map.
In the bilinear case, where $\mathfrak{M}=0$ holds, the kernels $h_{C,i}$ and $\widehat{h}_{C,i}$ are sufficient for the further analysis (see, e.g., \cite{zhang_2002,benner_birka_2012}). Similarly, in the LQO case, the authors in \cite{reiter_2024} defined only two kernels, one for the linear output of the form $h_{C}=C\bar{P}$ and one for the quadratic output of the form $h_{M}=\mathfrak{M}(\bar{P}\otimes\bar{P})$. The kernels were used to define $\sH_2$ inner products and $\sH_2$ norms for bilinear or LQO systems, respectively.

Generalizing the concepts for bilinear and LQO systems, we can
define an $\sH_2$ inner product and also an $\sH_2$ norm for BQO systems.
\begin{definition} \label{def:H2skp}
    Let $\mathcal{S}=(A,B,C,\mathcal{N},\mathcal{M})$ and $\widehat{\mathcal{S}}=(\widehat{A},\widehat{B},\widehat{C},\widehat{\mathcal{N}},\widehat{\mathcal{M}})$ be BQO systems from \cref{eq:BQO} and \cref{eq:BQOr}. Moreover, let $h_{C,i}, h_{M,i,j}, \widehat{h}_{\widehat{C},i}, \widehat{h}_{\widehat{M},i,j}$ be their Volterra kernels defined in \cref{eq:volt_kern}. 
    The $\sH_2$ inner product of $\mathcal{S}$ and $\widehat{\mathcal{S}}$ is defined as
    \begin{align*}
        \skp{\mathcal{S}}{\widehat{\mathcal{S}}}_{\sH_2}&:=\sum_{i=1}^\infty \int_0^\infty \dots\int_0^\infty \tr\left(h_{C,i}\widehat{h}_{\widehat{C},i}^T\right) dt_1\dots dt_i \\
        &\qquad + \sum_{i,j=1}^\infty \int_0^\infty \dots \int_0^\infty \tr\left( h_{M,i,j}\widehat{h}_{\widehat{M},i,j}^T\right) dt_1\dots dt_i d\bar{t}_1 \dots d\bar{t}_j,
    \end{align*}
    assuming all integrals exist and the infinite series converge.
    If $\mathcal{S}=\widehat{\mathcal{S}}$, then the $\sH_2$ norm of $\mathcal{S}$ is defined as 
    \begin{align*}
        \norm{\mathcal{S}}_{\sH_2}^2 &:= \sum_{i=1}^\infty \int_0^\infty \dots\int_0^\infty \norm{h_{C,i}}^2_F dt_1\dots dt_i \\
        &\qquad + \sum_{i,j=1}^\infty \int_0^\infty \dots \int_0^\infty \norm{h_{M,i,j}}^2_F dt_1\dots dt_i d\bar{t}_1 \dots d\bar{t}_j,
    \end{align*}
    assuming all integrals exist and the infinite series converge.
\end{definition}

\Cref{def:H2skp} introduces the $\sH_2$ inner product and $\sH_2$ norm for BQO systems. However, its computation is not clear from the definition. 
Similar to the bilinear and the LQO cases  in \cite{gugercin_2008,reiter_2024}, next we establish an explicit direct formula for the $\sH_2$ inner product and the $\sH_2$ norm of BQO systems using solutions to the generalized Lyapunov and Sylvester equations.
\begin{theorem}\label{theo:normS}
    Let $\mathcal{S}$ and $\widehat{\mathcal{S}}$ be the full and reduced BQO systems in~\eqref{eq:BQO} and~\eqref{eq:BQOr} with stable $A$ and $\widehat{A}$. Assume that $\norm{\mathcal{N}}^2,\norm{\mathcal{N}_T}^2<2\alpha/\beta^2$ and $\norm{\widehat{\mathcal{N}}}^2,\norm{\widehat{\mathcal{N}}_T}^2<2\widehat{\alpha}/\widehat{\beta}^2$ holds. 
    Then, the $\sH_2$ inner product $\skp{\mathcal{S}}{\widehat{\mathcal{S}}}_{\sH_2}$ is given by
    \begin{align} \label{eq:H2innerbqo}
        \skp{\mathcal{S}}{\widehat{\mathcal{S}}}_{\sH_2}&= \tr(CX\widehat{C}^T)+\sum_{j=1}^p\tr(X^TM_jX\widehat{M}_j) = -\tr(B^TY\widehat{B}),
    \end{align}
    where  $X$ is as defined in \cref{eq:sumxi} and \eqref{eq:P_mix}, and $Y$ in \cref{eq:Q_mix}.
    If $\mathcal{S}=\widehat{\mathcal{S}}$, then $X=P$ from~\cref{eq:Pfull} and $Y=-Q$ from~\cref{eq:Qfull}, and the $\sH_2$ norm of $\mathcal{S}$ is given by
    \begin{align}  \label{eq:H2normbqo}
        \norm{\mathcal{S}}_{\sH_2}^2&= \tr(CPC^T)+ \sum_{j=1}^p\tr(PM_jPM_j) = \tr(B^TQB).
    \end{align}  
\end{theorem}
\begin{proof}
Our proof follows the machinery of the proof \cite{benner_h2QB_2018} where a similar result for the $\sH_2$ norm was derived for a different class of systems, namely the quadratic-bilinear control systems. To simplify the notation, we omit writing the time variables explicitly in the various functions. Starting with the definition of the $\sH_2$ inner product in \Cref{def:H2skp}, we get
 \begin{align*}
     \skp{\mathcal{S}}{\widehat{\mathcal{S}}}_{\sH_2}&= \sum_{i=1}^\infty \int_0^\infty \dots\int_0^\infty \tr\left(h_{C,i}\widehat{h}_{\widehat{C},i}^T\right) dt_1\dots dt_i \\
        &~~~ + \sum_{i,j=1}^\infty \int_0^\infty \dots \int_0^\infty \tr\left( h_{M,i,j}\widehat{h}_{\widehat{M},i,j}^T\right) dt_1\dots dt_i d\bar{t}_1 \dots d\bar{t}_j. 
\end{align*}
Substituting the definition of the kernels from \cref{eq:volt_kern} into this formula leads to
\begin{align*}
        \skp{\mathcal{S}}{\widehat{\mathcal{S}}}_{\sH_2}
        &=\sum_{i=1}^\infty \int_0^\infty \dots\int_0^\infty \tr\left(C\bar{P}_i\widehat{\bar{P}}^T_i\widehat{C}^T\right) dt_1\dots dt_i \\
        &~~~ + \sum_{i,j=1}^\infty \int_0^\infty \dots \int_0^\infty \tr\left( \mathfrak{M}(\bar{P}_i\otimes \bar{P}_j)(\widehat{\bar{P}}_i \otimes \widehat{\bar{P}}_j)^T\widehat{\mathfrak{M}}^T\right) dt_1\dots dt_i d\bar{t}_1 \dots d\bar{t}_j\\ 
        &=\tr\left(CX\widehat{C}^T\right) +\tr\left( \mathfrak{M}(X\otimes X)\widehat{\mathfrak{M}}^T\right),
 \end{align*}
 where we use \cref{eq:sumxi} in the last step.
    Now, using \Cref{lem:id_M} we obtain that 
    \begin{align*}
    \skp{\mathcal{S}}{\widehat{\mathcal{S}}}_{\sH_2}&=\tr\left(CX\widehat{C}^T\right) +\sum_{j=1}^p \tr(X^TM_jX\widehat{M}_j) 
    =\tr\left(X^T\left(C^T\widehat{C}+\sum_{j=1}^pM_jX\widehat{M}_j\right)\right)\\
    &= - \tr\left( Y^TB\widehat{B}^T\right)
    = -\tr\left(B^TY\widehat{B}\right),
    \end{align*}
    where we apply \Cref{lem:sylv_tr} to \cref{eq:PQ_mix} in the last steps.    
    The second statement for $\norm{\mathcal{S}}_{\sH_2}$ follows directly with $\mathcal{S}=\widehat{\mathcal{S}}$.
\end{proof}

\begin{remark}
We recall that bilinear systems are received by setting $M_j=0$ in \cref{eq:BQO_output} and LQO systems are received by setting $N_k=0$ in \cref{eq:BQO_state} similarly. Thus, we see that our definition of the $\sH_2$ norm for BQO systems recovers the definitions for bilinear systems~\cite{benner_birka_2012}, and LQO systems~\cite{reiter_2024} as special cases. 
\end{remark}
The $\sH_2$ norm formula~\eqref{eq:H2normbqo} immediately leads to an explicit expression to measure the  $\sH_2$ distance between two BQO systems.
\begin{corollary}\label{cor:H2err}
  Under the assumptions of Theorem \ref{theo:normS},  it holds
    \begin{align*}
        \norm{\mathcal{S}-\widehat{\mathcal{S}}}_{\sH_2}^2 &= \tr(CPC^T)-2\tr(CX\widehat{C}^T)+\tr(\widehat{C}\widehat{P}\widehat{C}^T) \\
        &\qquad \qquad+ \sum_{j=1}^p \tr(PM_jPM_j)-2\tr(X^TM_jX\widehat{M}_j)+\tr(\widehat{P}\widehat{M}_j\widehat{P}\widehat{M}_j) \\
        &=\tr(B^TQB)  +2\tr(B^TY\widehat{B})+\tr(\widehat{B}^T\widehat{Q}\widehat{B}).
    \end{align*}
\end{corollary}
\begin{proof}
Follows directly due to 
       $ \norm{\mathcal{S}-\widehat{\mathcal{S}}}_{\sH_2}^2 = \norm{\mathcal{S}}_{\sH_2}^2-2\skp{\mathcal{S}}{\widehat{\mathcal{S}}}_{\sH_2}+\norm{\widehat{\mathcal{S}}}_{\sH_2}^2.$
\end{proof}

\subsection{Bounding the output via the $\sH_2$ norm} \label{sec:5}
A central reason the $\mathcal{H}_2$ norm plays such a critical role 
in model reduction of linear dynamical systems is that it directly 
bounds the norm of the output~\cite{ant_beattie_gug_2020}. This 
result was recently extended to bilinear systems~\cite{
redmann_bilinear_2021}. In this section, we establish the analogous 
theory for BQO systems.

Assume $u\in L^2$, i.e.,
    \begin{align*}
        \norm{u}_{L^2}^2 := \int_0^\infty \norm{u(s)}_2^2 ds = \int_0^\infty u^T(s)u(s) ds <\infty.
    \end{align*}
Next, we define the vector of the control components which correspond to a nonzero $N_k$, $k=1,\dots,m$:
\begin{align*}
    u^0=\begin{bmatrix}
        u_1^0 & \dots u_m^0
    \end{bmatrix}^T \text{ with } u_k^0=\begin{cases}
0, & \text{if } N_k=0, \\
u_k, & \text{else}.
\end{cases}
\end{align*}
Note that $\norm{u^0}_{L^2}\leq \norm{u}_{L^2}$. 

The output of the BQO system~\cref{eq:BQO} has two components as shown in~\eqref{eq:BQO_output}: namely the linear term $Cx(t)$ and the quadratic term 
$\mathfrak{M}(x(t) \otimes x(t))$. Since BQO systems share the same state equation with bilinear systems, we will first use a result from~\cite{
redmann_bilinear_2021} to bound the linear term $Cx(t)$.
\begin{theorem}{\cite[Theorem 4.1]{redmann_bilinear_2021}}\label{thm:redmann}
    Consider a bilinear system with linear output $y(t) = Cx(t)$, i.e. setting $M_j=0$ in \cref{eq:BQO}. Also suppose the initial condition $x_0=0$ and suppose that $P$ is the unique solution to \cref{eq:Pfull}.    
    Then it holds that
        \begin{align*}
            \sup_{t\geq 0} \norm{Cx(t)}_2 \leq (\tr(CPC^T))^{\frac{1}{2}} e^{ 0.5 \norm{u^0}_{L^2}^2} \norm{u}_{L^2}.
        \end{align*}
\end{theorem}
Now with this result, we are ready to bound the full quadratic output $y(t) = Cx(t) + \mathfrak{M}(x(t) \otimes x(t))$ for BQO systems inspired by \Cref{thm:redmann}.  
\begin{theorem}\label{thm:bqo_out_est}
    Consider the BQO system \cref{eq:BQO} with initial state $x_0=0$ and the unique solution $P$ to \cref{eq:Pfull}.  Then, the output $y(t)$ of the BQO system satisfies
    \begin{align*}
            \sup_{t\geq 0} \norm{y(t)}_2 &\leq(\tr(CPC^T))^{\frac{1}{2}} e^{ 0.5 \norm{u^0}_{L^2}^2} \norm{u}_{L^2}\\
            &\qquad +\left(\sum_{j=1}^p\tr(PM_jPM_j)\right)^{\frac{1}{2}} e^{ \norm{u^0}_{L^2}^2} \norm{u\otimes u}_{L^2}.
        \end{align*}
\end{theorem}
\begin{proof}
    First, we see with \cref{eq:BQO_output} that 
    \begin{align}
        \norm{y(t)}_2 = \norm{Cx(t)+\mathfrak{M}(x(t)\otimes x(t))}_2 
        \leq \norm{Cx(t)}_2 + \norm{\mathfrak{M}(x(t)\otimes x(t))}_2. \label{eq:y_norm}
    \end{align}
    We focus on the second summand because the first term can be treated using
\Cref{thm:redmann}. 
    Following the discussion in \cite[Section 2]{redmann_bilinear_2021} (in particular,  \cite[Theorem 2.1]{redmann_bilinear_2021}), the solution $x(t)$ of the state equation \eqref{eq:BQO_state} can be written as  $x(t)=\Phi_u(t,0)x_0+\int_0^t \Phi_u(t,s)Bu(s) ds$ with the 
    fundamental solution $\Phi_u(t,s)$ for $s\leq t$ of \cref{eq:BQO_state}, which solves 
    \begin{align*}
        \Phi_u(t,s)=I+\int_s^tA\Phi_u(\tau,s) d\tau + \sum_{k=1}^m\int_s^tN_k\Phi_u(\tau,s)u_k(\tau)d\tau.
    \end{align*}
    The expression for $x(t)$ simplifies for the zero initial condition $x_0=0$ to $x(t)=\int_0^t \Phi_u(t,s)Bu(s) ds$.
    Inserting this expression in the second term in \cref{eq:y_norm} yields
    \begin{align}
        \norm{\mathfrak{M}(x(t)\otimes x(t))}_2 
       &= \norm{\mathfrak{M}\left(\int_0^t \Phi_u(t,s_1)Bu(s_1)ds_1 \otimes \int_0^t \Phi_u(t,s_2)Bu(s_2)ds_2\right)}_2\nonumber\\
        & \leq \int_0^t \int_0^t \norm{\mathfrak{M}\left( \Phi_u(t,s_1)Bu(s_1) \otimes  \Phi_u(t,s_2)Bu(s_2)\right)}_2 ds_1 ds_2\nonumber \\
        & = \int_0^t \int_0^t\norm{\mathfrak{M}\left( \Phi_u(t,s_1)B \otimes  \Phi_u(t,s_2)B\right)(u(s_1)\otimes u(s_2))}_2 ds_1 ds_2\nonumber.
        \end{align}
        This can be further estimated as
        \begin{align}
        & \norm{\mathfrak{M}(x(t)\otimes x(t))}_2 \nonumber\\        
        &\qquad \leq \int_0^t \int_0^t \norm{\mathfrak{M}\left( \Phi_u(t,s_1)B \otimes  \Phi_u(t,s_2)B\right)}_F \norm{u(s_1)\otimes u(s_2)}_2 ds_1 ds_2\nonumber\\
        &\qquad \leq \left(\int_0^t \int_0^t \norm{\mathfrak{M}\left( \Phi_u(t,s_1)B \otimes  \Phi_u(t,s_2)B\right)}_F^2 ds_1 ds_2\right)^{\frac{1}{2}} \label{Mxxnorm1} \\
        &\qquad\qquad\qquad\times \left(\int_0^t \int_0^t\norm{u(s_1)\otimes u(s_2)}^2_2 ds_1 ds_2\right)^{\frac{1}{2}}. \nonumber 
    \end{align}
    First, let us focus on the first integral. For the integrand, we have 
    \begin{align}
      &  \norm{\mathfrak{M}\left( \Phi_u(t,s_1)B \otimes  \Phi_u(t,s_2)B\right)}^2_F \nonumber\\
      &\qquad = \tr(\left(\left(\mathfrak{M}( \Phi_u(t,s_1)B \otimes  \Phi_u(t,s_2)B\right)\right) \left[\mathfrak{M}\left( \Phi_u(t,s_1)B \otimes  \Phi_u(t,s_2)B\right)\right)^T)\nonumber\\
      &\qquad  = \tr(\left(\mathfrak{M}\left( \Phi_u(t,s_1)B \otimes  \Phi_u(t,s_2)B\right)\right) \left(\left( B^T\Phi^T_u(t,s_1) \otimes  B^T\Phi^T_u(t,s_2)\right)\mathfrak{M}^T\right))\nonumber\\
      &\qquad  = \tr(\mathfrak{M}\left( \Phi_u(t,s_1)BB^T\Phi^T_u(t,s_1) \otimes  \Phi_u(t,s_2)BB^T\Phi^T_u(t,s_2)\right) \mathfrak{M}^T) \label{eq:Mppnorm}.
    \end{align}
 In order to bound the term $ \Phi_u(t,s)BB^T\Phi^T_u(t,s)$, following \cite[Lemma 2.3]{redmann_bilinear_2021}, we define
    $\bar{Z}(t,Z_0),t\geq 0$ as the solution to 
    \begin{align*}
        \dot{\bar{Z}}(t)=A\bar{Z}(t) + \bar{Z}(t)A^T + \sum_{k=1}^m N_k\bar{Z}(t)N_k^T, \quad \bar{Z}(0) = Z_0\geq 0.
    \end{align*}   
    Using this definition, we know that 
    \begin{align*}
        \Phi_u(t,s)BB^T\Phi^T_u(t,s) \leq e^{ \int_0^t\norm{u^0(\tau)}_2^2d\tau } \bar{Z}(t-s,BB^T)
    \end{align*}
    holds~\cite[Theorem 4.1]{redmann_bilinear_2021}. With \cref{eq:Mppnorm}, we use this expression to estimate \cref{Mxxnorm1} and obtain
    \begin{align*}
        &\int_0^t \int_0^t \norm{\mathfrak{M}\left( \Phi_u(t,s_1)B \otimes  \Phi_u(t,s_2)B\right)}_F^2 ds_1 ds_2\\
        &~~~=
        \int_0^t \int_0^t \tr(\mathfrak{M}\left( \Phi_u(t,s_1)BB^T\Phi^T_u(t,s_1) \otimes  \Phi_u(t,s_2)BB^T\Phi^T_u(t,s_2)\right) \mathfrak{M}^T) ds_1 ds_2\\
        &~~~\leq  e^{2\int_0^t\norm{u^0(\tau)}_2^2d\tau }
        \int_0^t \int_0^t \tr\left(\mathfrak{M}\left(\bar{Z}(t-s_1,BB^T)\otimes \bar{Z}(t-s_2,BB^T)\right) \mathfrak{M}^T\right) ds_1 ds_2\\
        &~~~\leq  e^{2\int_0^t\norm{u^0(\tau)}_2^2d\tau }  \tr\left(\mathfrak{M} \int_0^t \int_0^t\bar{Z}(s_1,BB^T)\otimes \bar{Z}(s_2,BB^T)  ds_1 ds_2 \mathfrak{M}^T\right) \\
       &~~~ \leq  e^{2\int_0^t\norm{u^0(\tau)}_2^2d\tau } \tr(\mathfrak{M}(P_t\otimes P_t) \mathfrak{M}^T) 
    \end{align*}
    with $P_t=\int_0^t \bar{Z}(s,BB^T) ds$. Next, it holds $P=\lim_{t\to\infty}P_t$ (see proof of \cite[Theorem 4.1]{redmann_bilinear_2021}). Combining this estimate with \cref{eq:y_norm,Mxxnorm1} and the estimate from \cref{thm:redmann} for $\norm{Cx(t)}$, and taking the supremum, we obtain
    \begin{align*}
        \sup_{t\geq 0}\norm{y(t)}_2 &\leq (\tr(CPC^T))^{\frac{1}{2}} e^{ 0.5 \norm{u^0}_{L^2}^2} \norm{u}_{L^2}\\
        &\qquad + (\tr(\mathfrak{M}(P\otimes P) \mathfrak{M}^T))^{\frac{1}{2}}  e^{  \norm{u^0}_{L^2}^2} \norm{u\otimes u}_{L^2}.
    \end{align*}
    Applying \Cref{lem:id_M}  again yields the desired result. 
\end{proof}

Now, based on Theorem \ref{theo:normS}, we can estimate on the output using the $\sH_2$ norm from \Cref{def:H2skp}. More precisely, we loosen the estimate on the output but this enables us to show a direct connection to the $\sH_2$ norm involved here.
\begin{theorem} \label{cor:yest}
Consider a BQO system $\mathcal{S}$ as in \cref{eq:BQO} with initial state $x_0=0$ and suppose that $P$ is the unique solution to \cref{eq:Pfull}. Then it holds further
\begin{align*}
 \sup_{t\geq 0} \norm{y(t)}^2_2 \leq& 
 \norm{\mathcal{S}}^2_{\mathcal{H}_2} 
 \left(e^{  \norm{u^0}_{L^2}^2} \norm{u}^2_{L^2} + e^{ 2\norm{u^0}_{L^2}^2} \norm{u\otimes u}^2_{L^2}\right).
        \end{align*}
\end{theorem}
\begin{proof}
    We see with \Cref{thm:bqo_out_est}, that
    \begin{align*}
            \norm{y(t)}_2 &\leq \underbrace{(\tr(CPC^T))^{\frac{1}{2}}}_{=:T_1} \underbrace{e^{ 0.5 \norm{u^0}_{L^2}^2} \norm{u}_{L^2}}_{=:U_1} +\underbrace{\left(\sum_{j=1}^p\tr(PM_jPM_j)\right)^{\frac{1}{2}}}_{=:T_2} \underbrace{e^{ \norm{u^0}_{L^2}^2} \norm{u\otimes u}_{L^2}}_{=:U_2}
        \end{align*}
    where $T_j$ represent the trace-terms, while $U_j$ represent the exp-times-norm-terms.
    From this, we have  $\norm{y(t)}^2_2 \leq\left(T_1U_1+T_2U_2\right)^2=T_1^2U_1^2+T_2^2U_2^2+2T_1U_1T_2U_2.$
Moreover, from $0\leq(T_1U_2-T_2U_1)^2=T_1^2U_2^2+T_2^2U_1^2-2T_1U_1T_2U_2,$ we have $2T_1U_1T_2U_2 \leq T_1^2U_2^2+T_2^2U_1^2.$
        This gives
        \begin{align*}
            \norm{y(t)}^2_2 &\leq T_1^2U_1^2+T_2^2U_2^2+T_1^2U_2^2+T_2^2U_1^2 = (T_1^2+T_2^2)(U_1^2+U_2^2)
        \end{align*}
        which proves the desired statement after taking the supremum on the left side.
\end{proof}
With this result, we can finally state a bound on the model reduction output error $\norm{y(t)-\widehat{y}(t)}_2$ using the $\sH_2$ distance.    
\begin{corollary} \label{cor:yerr_est}
    Consider the full BQO system $\mathcal{S}$ \cref{eq:BQO} and the reduced BQO system $\widehat{\mathcal{S}}$ \cref{eq:BQOr} with initial states $x_0=0$ and $\widehat{x}_0=0$, respectively. Furthermore, suppose that $P,\widehat{P}, X$ are the unique solutions to \cref{eq:Pfull,eq:P_mix,eq:P_red}. Then the output error $y(t)-\widehat{y}(t)$ satisfies
    \begin{align*}
        &\sup_{t\geq 0} \norm{y(t)-\widehat{y}(t)}^2_2 \leq \norm{\mathcal{S}-\widehat{\mathcal{S}}}_{\mathcal{H}_2}^2 
        \left(e^{ \norm{u^0}_{L^2}^2} \norm{u}^2_{L^2} 
            + e^{ 2\norm{u^0}_{L^2}^2} \norm{u\otimes u}^2_{L^2}\right).
    \end{align*}
\end{corollary}
\begin{proof}
    Follows directly from \Cref{cor:yest} with $\mathcal{S}_e=\mathcal{S}-\widehat{\mathcal{S}}$.
\end{proof}

\section{$\sH_2$ optimal reduction for BQO systems}
\label{sec:6}
    In this section, we discuss model order reduction in an optimal $\sH_2$ sense. \Cref{cor:yerr_est} gives us an estimate for the output error due to model reduction. Hereby, the output error is dominated by a constant term consisting of the input norms and the squared $\sH_2$ system error $\norm{\mathcal{S}-\widehat{\mathcal{S}}}_{\sH_2}^2$. 
    Hence, the objective of model order reduction in the $\sH_2$ sense is to determine a reduced system $\widehat{\mathcal{S}}$ that minimizes the $\sH_2$ approximation error for a given system $\mathcal{S}$. In view of \Cref{cor:yerr_est}, this ensures that the corresponding output error is small.

    \subsection{$\sH_2$ optimality for BQO systems}
In this context, we seek conditions under which the error $\norm{\mathcal{S} - \widehat{\mathcal{S}}}_{\mathcal{H}_2}^2$ is minimized. 
Based on the $\sH_2$ analysis of \Cref{sec:H2analysys}, we are now in a position to state a theorem characterizing locally $\sH_2$ optimal reduced systems which fulfill the first-order necessary optimality conditions. 
\begin{theorem}\label{thm:grad}
 Consider the full BQO system $\mathcal{S}$ \cref{eq:BQO} and assume that the reduced BQO system $\widehat{\mathcal{S}}$ \cref{eq:BQOr} is locally $\sH_2$ optimal. Furthermore, let $\widehat{P},X$ be the unique solutions to \cref{eq:P_red,eq:P_mix} and also $\Psi=\Psi^T\in\R^{r\times r}$ and $\Pi\in\R^{n\times r}$ be the unique solutions to the generalized Lyapunov and Sylvester equations
    \begin{subequations}
    \begin{align}
        \widehat{A}^T\Psi+\Psi\widehat{A}+\sum_{k=1}^m\widehat{N}_k^T\Psi\widehat{N}_k+2\sum_{j=1}^p\widehat{M}_j\widehat{P}\widehat{M}_j+\widehat{C}^T\widehat{C}&=0, \label{eq:lyap_pi1} \\
        A^T\Pi+\Pi\widehat{A}+\sum_{k=1}^mN_k^T\Pi\widehat{N}_k+2\sum_{j=1}^pM_jX\widehat{M}_j+C^T\widehat{C}&=0. \label{eq:lyap_pi2}
    \end{align}
     \end{subequations}
    Then  the following first-order optimality conditions hold  \\
    \begin{subequations}\label{eq:optcond}
    \noindent\begin{minipage}{.5\linewidth}    
    \begin{align}
        0 &= \Psi\widehat{P}-\Pi^TX, \label{eq:optcond1}\\
        0 &= \Psi\widehat{B}-\Pi^TB,\label{eq:optcond2}\\
        0 &= \Psi\widehat{N}_k\widehat{P}-\Pi^TN_kX,\label{eq:optcond3}        
    \end{align}    
    \end{minipage}
    \begin{minipage}{.45\linewidth}
    \begin{align}
    0 &= \widehat{C}\widehat{P}-CX,\label{eq:optcond4}\\
        0 &= \widehat{P}\widehat{M}_j\widehat{P}-X^TM_jX,\label{eq:optcond5}
    \end{align}
    \end{minipage}
    \end{subequations}  \\ 

    \noindent
    for all $k=1,\dots, m$ and $j=1,\dots, p$.
\end{theorem}

\begin{proof}
Note that the reduced model $\widehat{\mathcal{S}}$ is fully characterized by its state-space matrices 
$\widehat{A}, \widehat{B}, \widehat{C}, \widehat{N}_k, \widehat{M}_j$. \Cref{cor:H2err} gives an explicit formula for the $\sH_2$ error based on these matrices and the matrices $\widehat{P},X$ as solutions to \cref{eq:P_red,eq:P_mix}. Therefore, following the approach from \cite[Section~2.4.4.2]{mlinaric_structure-preserving} for structured linear dynamical systems, the $\sH_2$ optimal model reduction problem for BQO systems can be equivalently written as the constrained optimization problem
\begin{align*}
    \min_{\widehat{A}, \widehat{B}, \widehat{C}, \widehat{N}_k, \widehat{M}_j} \;& 
    \tr(\widehat{C}\widehat{P}\widehat{C}^T) - 2\tr(CX\widehat{C}^T) 
    + \sum_{j=1}^p \tr(\widehat{P}\widehat{M}_j \widehat{P}\widehat{M}_j) 
    - 2\tr(X^T M_j X \widehat{M}_j), \\
    \text{subject~to~} \;& \cref{eq:P_red} \ \text{and} \ \cref{eq:P_mix}.
\end{align*}
Here in this formulation, the constant terms in $\norm{\mathcal{S} - \widehat{\mathcal{S}}}_{\mathcal{H}_2}^2$
from~\Cref{cor:H2err} have been omitted.
    The associated Lagrange function is given as
    \begin{align*}
        &\mathfrak{L}(\widehat{A},\widehat{B},\widehat{C},\widehat{N}_k,\widehat{M}_j,\widehat{P},X,\Lambda_1,\Lambda_2) \\
        &= 
        \tr(\widehat{C}\widehat{P}\widehat{C}^T)-2\tr(CX\widehat{C}^T) +\sum_{j=1}^p\tr(\widehat{P}\widehat{M}_j\widehat{P}\widehat{M}_j)-2\tr(X^TM_jX\widehat{M}_j) \\
        &\qquad\qquad+ \tr\left(\Lambda_1^T\left(\widehat{A}\widehat{P}+\widehat{P}\widehat{A}^T+\sum_{k=1}^m\widehat{N}_k\widehat{P}\widehat{N}_k^T+\widehat{B}\widehat{B}^T\right)\right) \\
        &\qquad\qquad+ \tr\left(\Lambda_2^T\left(AX+X\widehat{A}^T+\sum_{k=1}^mN_kX\widehat{N}_k^T+B\widehat{B}^T\right)\right)
    \end{align*}
    with the Lagrange multipliers $\Lambda_1\in\R^{r\times r}$ and $\Lambda_2\in\R^{n\times r}$.
    Next, we compute the gradients of $\mathfrak{L}$ using \Cref{lem:gradtr} and employing the trace identities $\tr(AB)=\tr(BA)$ and $\tr(A)=\tr(A^T)$ and  the symmetry of $\widehat{P}$, $M_j$, and $\widehat{M}_j$. After some algebraic manipulations, we obtain
       \begin{align}
        \nabla_{\widehat{C}}\mathfrak{L} &=2\widehat{C}\widehat{P}-2CX,\label{eq:gradC} \\
        \nabla_{\widehat{M}_j}\mathfrak{L} &=2\widehat{P}\widehat{M}_j\widehat{P}-2 X^TM_jX,\label{eq:gradM} \\
        \nabla_{\widehat{P}}\mathfrak{L} = \widehat{C}^T\widehat{C}+2\sum_{j=1}^p\widehat{M}_j\widehat{P}\widehat{M}_j+&\Lambda_1^T\widehat{A}+\widehat{A}^T\Lambda_1^T + \sum_{k=1}^m\widehat{N}_k^T\Lambda_1^T\widehat{N}_k,\nonumber \\
        \nabla_{X}\mathfrak{L} = -2C^T\widehat{C}-2\sum_{j=1}^pM_jX\widehat{M}_j-&2\sum_{j=1}^pM_{j}X\widehat{M}_j + A^T\Lambda_2+\Lambda_2\widehat{A}+\sum_{k=1}^mN_k^T\Lambda_2\widehat{N}_k.  \nonumber      
    \end{align}
    Then, \cref{eq:optcond4,eq:optcond5} follow directly with setting \cref{eq:gradC,eq:gradM} to zero. Next, $\nabla_{\widehat{P}}\mathfrak{L}=0$ and $\nabla_{X}\mathfrak{L}=0$ yield $\Lambda_1=\Lambda_1^T=\Psi$ from \cref{eq:lyap_pi1} and $\Lambda_2=-2\Pi$ from \cref{eq:lyap_pi2}. 
    It remains to compute the remaining gradients:
    \begin{align*}
        \nabla_{\widehat{A}}\mathfrak{L} &= \Lambda_1\widehat{P}+\Lambda_1\widehat{P}+ \Lambda_2^TX = 2\Psi\widehat{P}-2\Pi^TX,\\
        \nabla_{\widehat{B}}\mathfrak{L} &= \Lambda_1\widehat{B}+\Lambda_1\widehat{B}+ \Lambda_2^TB = 2\Psi\widehat{B}-2\Pi^TB,\\
        \nabla_{\widehat{N}_k}\mathfrak{L} &= \Lambda_1\widehat{N}_k\widehat{P}+\Lambda_1\widehat{N}_k\widehat{P}+ \Lambda_2^TN_kX = 2\Psi\widehat{N}_k\widehat{P}-2\Pi^TN_kX.
    \end{align*}    
    Then setting $\nabla_{\widehat{A}}\mathfrak{L}=0$, $\nabla_{\widehat{B}}\mathfrak{L}=0$, and $\nabla_{\widehat{N}_k}\mathfrak{L}=0$ yields 
    \cref{eq:optcond1,eq:optcond2,eq:optcond3}.
\end{proof}

Having derived the necessary optimality conditions for minimizing the $\sH_2$ error, the next important question we want to answer is whether the reduced order quantities satisfying optimality conditions of~\Cref{thm:grad} can be obtained via a Petrov-Galerkin projection as in~\eqref{eq_PetrovGalerkin}. The following theorem characterizes the construction of an $\sH_2$ optimal reduced system from this projection perspective.
\begin{theorem} \label{thm:redsysproj}
    Consider the full BQO system $\mathcal{S}$ \cref{eq:BQO} and assume that the reduced BQO system $\widehat{\mathcal{S}}$ \cref{eq:BQOr} is a locally $\sH_2$ optimal approximant. Furthermore, suppose that $\widehat{P},X$ are the unique solutions to \cref{eq:P_red,eq:P_mix} and also suppose that $\Psi$ and $\Pi$ are the unique solutions to \cref{eq:lyap_pi1,eq:lyap_pi2}. If $\widehat{P}$ and $\Psi$ are invertible, the optimal reduced BQO system can be obtained via a Petrov--Galerkin-projection as in~\eqref{eq_PetrovGalerkin} 
    with $V=X\widehat{P}^{-1}$ and $W=\Pi\Psi^{-1}$.
\end{theorem}
\begin{proof}
    Since $\widehat{P}$ and $\Psi$ are invertible, we can define $V:=X\widehat{P}^{-1}$ and $W:=\Pi\Psi^{-1}$. Using \cref{eq:optcond1}, we see $W^TV=\Psi^{-1}\Pi^TX\widehat{P}^{-1}= I_r$. Next, \cref{eq:optcond2} yields $\widehat{B}=\Psi^{-1}\Pi^TB=W^TB$ and \cref{eq:optcond3} yields $\widehat{N}_k= \Psi^{-1}\Pi^TN_kX\widehat{P}^{-1}= W^TN_kV$. Similarly, \cref{eq:optcond4} leads to $\widehat{C}=CX\widehat{P}^{-1}$ and \cref{eq:optcond5} leads to $\widehat{M}_j=\widehat{P}^{-1}X^TM_jX\widehat{P}^{-1}=V^TM_jV$. 
    Finally, using $X = V\widehat{P}$, we left-multiply \cref{eq:P_mix} by $W^T$ to obtain
    \begin{align*}
         0&= W^TAV\widehat{P}+W^TV\widehat{P}\widehat{A}^T + \sum_{k=1}^mW^TN_kV\widehat{P}\widehat{N}_k^T+ W^TB\widehat{B}^T\\
        &= \left(W^TAV\right)\widehat{P}+\widehat{P}\widehat{A}^T + \sum_{k=1}^m\widehat{N}_k\widehat{P}\widehat{N}_k^T+ \widehat{B}\widehat{B}^T.
    \end{align*}
    Since we have assumed that the solution to \cref{eq:P_red} is unique, it follows $\widehat{A}=W^TAV$. It therefore follows from Theorem~\ref{thm:grad} that the reduced system considered here satisfies the optimality conditions~\eqref{eq:optcond}.
\end{proof}

Hence, \Cref{thm:redsysproj} provides us with a formulation of an $\sH_2$ optimal reduced system which can be obtained by a Petrov-Galerkin projection as in the linear~\cite{Wil70}, bilinear~\cite{benner_birka_2012}, and LQO cases~\cite{reiter_2024}. 
\Cref{thm:redsysproj} shows that the optimal Petrov-Galerkin projection is achieved by setting
 $V=X\widehat{P}^{-1}$ and $W=\Pi\Psi^{-1}$. In other words, the optimal reduced BQO system is given by $\widehat{\mathcal{S}}=(\widehat{A},\widehat{B},\widehat{C},\widehat{\mathcal{N}},\widehat{\mathcal{M}})= (W^TAV, W^TB, CV, \widehat{\mathcal{N}},\widehat{\mathcal{M}})$ with $\widehat{\mathcal{N}}=\begin{bmatrix} W^TN_1V \, \dots \, W^TN_mV \end{bmatrix}$ and  $\widehat{\mathcal{M}}=\begin{bmatrix} V^TM_1V \, \dots \, V^TM_pV \end{bmatrix}$.
Using a state-space transformation $\widehat{x}=\widehat{P}z$, inserting $V=X\widehat{P}^{-1}$ and $W=\Pi\Psi^{-1}$, and multiplying by $\Psi$ on the left, we receive the transformed system
    \begin{align*}
    \Psi\widehat{P}\dot{z}(t) &= \Psi W^T AV\widehat{P}z(t) + \sum_{k=1}^m \Psi W^T N_kV\widehat{P}z(t) u_k(t)+ \Psi W^T B u(t)  \\
    &= \Pi^T AXz(t) + \sum_{k=1}^m \Pi^T N_kXz(t) u_k(t)+ \Pi^T B u(t),  \\
    \widehat{y}(t)&= CV\widehat{P}z(t)+\begin{bmatrix}
        z(t)^T\widehat{P}V^TM_1V\widehat{P}z(t) \\
        \vdots \\
        z(t)^T\widehat{P}V^TM_pV\widehat{P}z(t)
    \end{bmatrix} 
    =CXz(t)+\begin{bmatrix}
        z(t)^TX^TM_1Xz(t) \\
        \vdots \\
        z(t)^TX^TM_pXz(t)
    \end{bmatrix}.
\end{align*}
Left-multiplying the first equation by $(\Psi\widehat P)^{-1}$ and making use of \eqref{eq:optcond1}, that is,  $\Psi\widehat{P}=\Pi^TX$, we obtain an equivalent representation of the optimal reduced system with system matrices 

\begin{subequations}\label{eq:red_mat}
\noindent\begin{minipage}{.5\linewidth}
    \begin{align}
    \widehat{A} &:=(\Pi^TX)^{-1}\Pi^TAX, \label{eq:red_mat1}\\
    \widehat{B} &:=(\Pi^TX)^{-1}\Pi^TB, \label{eq:red_mat2}\\
    \widehat{C} &:= CX, \label{eq:red_mat3}
    \end{align}
    \end{minipage}
    \begin{minipage}{.45\linewidth}
    \begin{align}
    \widehat{N}_k &:= (\Pi^TX)^{-1}\Pi^TN_kX, \label{eq:red_mat4}\\
    \widehat{M}_j &:= X^TM_jX, \label{eq:red_mat5}
    \end{align}
    \end{minipage}
\end{subequations}\\ 

    \noindent
    for all $k=1,\dots, m$ and $j=1,\dots, p$. As~\eqref{eq:red_mat} illustrates, to construct the optimal reduced model via projection, we only need $X$ and $\Pi$ and do not need to compute $\widehat{P}$ and $\Psi$ explicitly. 

We have thus established the theoretical foundation for $\sH_2$ optimal model reduction via a Petrov--Galerkin projection. A key difficulty, however, is that the matrices $\Pi$ and $X$ needed to construct the optimal reduced model as in~\eqref{eq:red_mat} must be computed as solutions to \cref{eq:P_mix,eq:lyap_pi2}, which depend on the reduced system itself. 
We will address this issue an iterative projection procedure, where the system is repeatedly updated until convergence, i.e., until the optimality conditions are satisfied. This approach is discussed in detail in the next section.

\subsection{BQO-TSIA}
As the discussion of the previous section reveals,  constructing a (locally) $\sH_2$ optimal
reduced BQO system requires an iterative algorithm. This is what we develop in this section which generalizes IRKA \cite{gugercin_2008} and TSIA \cite{xu_optimal_2011} for linear systems, BIRKA \cite[Algorithm 1]{benner_birka_2012} for bilinear systems and the LQO-TSIA \cite[Algorithm 1]{reiter_2024} for linear quadratic output systems to BQO systems. The resulting algorithm, namely \Cref{alg:BQO_TSIA}, enforces the optimality conditions upon convergence. 

\Cref{alg:BQO_TSIA} requires the system matrices of a BQO system of full order and those of an initial, reduced BQO system. In each iteration, $X\in\R^{n\times r}$ and $\Pi\in\R^{n\times r}$ are computed as solutions to \eqref{eq:alg1} in step 2. Based on $X$ and $\Pi$, a Petrov--Galerkin projection is applied to the full BQO system in step 4, according to \eqref{eq:red_mat}. Then, this new reduced system is used for the next iteration to compute a new projection. When a certain termination condition is reached, the algorithm stops and the last computed reduced BQO system is obtained as the output of the algorithm.  In the following theorem, we show that upon convergence of \Cref{alg:BQO_TSIA} the resulting reduced system fulfills the desired optimality conditions.

\begin{algorithm}
    \caption{\texttt{Bilinear quadratic output two-sided iteration algorithm (BQO-TSIA)}}\label{alg:BQO_TSIA}
    \begin{algorithmic}[1]
    \Require $A, B, C, N_k, M_j$ as in \eqref{eq:BQO}, $\widehat{A}, \widehat{B}, \widehat{C}, \widehat{N}_k, \widehat{M}_j$ as in \eqref{eq:BQOr}
    \Ensure $\widehat{A}^{opt}, \widehat{B}^{opt}, \widehat{C}^{opt}, \widehat{N}^{opt}_k, \widehat{M}^{opt}_j$ as in \eqref{eq:BQOr}
    \While{not converged}
        \State Solve 
        \begin{subequations}  \label{eq:alg1}     
        \begin{align}
            AX+X\widehat{A}^T+\sum_{k=1}^mN_kX\widehat{N}_k^T+B\widehat{B}^T&=0 \label{eq:alg1a}\\
            A^T\Pi+\Pi\widehat{A}+\sum_{k=1}^mN_k^T\Pi\widehat{N}_k +2\sum_{j=1}^pM_jX\widehat{M}_j + C^T\widehat{C}&=0 \label{eq:alg1b}
        \end{align}
        \end{subequations}
        for $X\in\R^{n\times r}$ and $\Pi\in\R^{n\times r}$.
        \State Compute $V=\orth(X)$, $W=\orth(\Pi)$.
        \State Compute $\widehat{A}=(W^TV)^{-1}W^TAV$, $\widehat{B}=(W^TV)^{-1}W^TB$, $\widehat{C}=CV$, $\widehat{N}_k=(W^TV)^{-1}W^TN_kV$, $\widehat{M}_j=V^TM_jV$.
    \EndWhile
    \State $\widehat{A}^{opt}=\widehat{A}, \widehat{B}^{opt}=\widehat{B}, \widehat{C}^{opt}=\widehat{C}, \widehat{N}^{opt}_k=\widehat{N}_k$ and $\widehat{M}_j^{opt}=\widehat{M}_j$. 
    \end{algorithmic}
    \end{algorithm} 
 
\begin{theorem}
    Assume \Cref{alg:BQO_TSIA} converges. Then, the reduced system $\widehat{\mathcal{S}}^{opt}$ given by $\widehat{A}^{opt}, \widehat{B}^{opt}, \widehat{C}^{opt}, \widehat{N}^{opt}_k$ for $k=1,\dots , m$ and $\widehat{M}_j^{opt}$ for $j=1,\dots, p$ fulfills the necessary $\mathcal{H}_2$ optimality conditions \cref{eq:optcond}.
\end{theorem}
\begin{proof}
    Let $\bar{A}$, $\bar{B}$, $\bar{C}$, $\bar{N}_k$, $\bar{M}_j$ be the matrices corresponding to the reduced system $\bar{\mathcal{S}}$ in the next to the last step of the algorithm before convergence. Then, $\widehat{\mathcal{S}}^{opt}$ is a state space transformation of $\bar{\mathcal{S}}$, i.e., there exists a nonsingular $T\in\R^{r\times r}$, such that
    \begin{align}\label{eq:sysbar}
        \bar{A}=T^{-1}\widehat{A}^{opt}T, \bar{B}=T^{-1}\widehat{B}^{opt}, \bar{C}=\widehat{C}^{opt}T, \bar{N}_k=T^{-1}\widehat{N}_k^{opt}T, \bar{M}_j=T^{T}\widehat{M}_j^{opt}T. 
    \end{align}
    Next, it holds in step 3 of \Cref{alg:BQO_TSIA} that
    $V^{opt}=X^{opt}F$ and $W^{opt}=\Pi^{opt}G$ with suitably chosen nonsingular matrices $F,G\in \R^{r\times r}$. It follows that
    \begin{align}\label{eq:WVWproj}
        ((W^{opt})^TV^{opt})^{-1}(W^{opt})^T= F^{-1}((\Pi^{opt})^TX^{opt})^{-1}(\Pi^{opt})^{T}.
    \end{align}
    We see from step 2 that
    \begin{align}
        AX^{opt}+X^{opt}\bar{A}^T+ \sum_{k=1}^m N_kX^{opt}\bar{N}_k^T+B\bar{B}^T=0. \label{eq:sylv_opt}
    \end{align}
    Using \cref{eq:WVWproj,eq:sysbar}, we left-multiply \cref{eq:sylv_opt} by $((W^{opt})^TV^{opt})^{-1}(W^{opt})^T$ and right-multiply by $T^T$ to obtain
    \begin{align*}
        0&=((W^{opt})^TV^{opt})^{-1}(W^{opt})^TAX^{opt}T^T+F^{-1}\bar{A}^T T^{T}\\
        &\qquad + \sum_{k=1}^m ((W^{opt})^TV^{opt})^{-1}(W^{opt})^T N_kX^{opt}\bar{N}_k^TT^T+\widehat{B}^{opt}\bar{B}^TT^T\\
        &=\widehat{A}^{opt}F^{-1}T^T+F^{-1}T^T(\widehat{A}^{opt})^T +\sum_{k=1}^m \widehat{N}_k^{opt}F^{-1}T^T(\widehat{N}_k^{opt})^T+\widehat{B}^{opt}(\widehat{B}^{opt})^T.
    \end{align*}
    Due to the uniqueness of the solution, we obtain $\widehat{P}=F^{-1}T^T$. 
    We can perform a similar procedure for the second equation in step 2 by multiplying it  by $(V^{opt})^T$ from the left and by $T^{-1}$ from the right to obtain $\Psi=F^T(X^{opt})^T\Pi^{opt}T^{-1}=T^{-T}(\Pi^{opt})^TX^{opt} F$, using the symmetry of $\Psi$. 
    Next, we again take a look at \cref{eq:sylv_opt} where we 
    substitute $\bar{A},\bar{N}_k,\bar{B}$ using \cref{eq:sysbar} and multiply by $T^T$ from the right to get
    \begin{align*}
        AX^{opt}T^T+X^{opt}T^T(\widehat{A}^{opt})^T+ \sum_{k=1}^m N_kX^{opt}T^T(\widehat{N}^{opt}_k)^T+B(\widehat{B}^{opt})^T=0.
    \end{align*}
    Due to the uniqueness of the solution, we have then $X=X^{opt}T^T$. Analogously for step 3, this yields $\Pi=\Pi^{opt}T^{-1}$.
    Now, we have expressions for $\widehat{P},\Psi,X$ and $\Pi$ which help us to show that the necessary $\sH_2$ optimality conditions in \cref{eq:optcond} are fulfilled:
    \begin{align*}
        \Psi\widehat{P}-\Pi^TX&=T^{-T}(\Pi^{opt})^TX^{opt}FF^{-1}T^T-T^{-T}(\Pi^{opt})^TX^{opt}T^T=0, \\
        \Psi\widehat{B}^{opt}-\Pi^TB&=T^{-T}(\Pi^{opt})^TX^{opt}F\widehat{B}^{opt}-T^{-T}(\Pi^{opt})^TB\\
        &=T^{-T}(\Pi^{opt})^TX^{opt}FF^{-1}((\Pi^{opt})^TX^{opt})^{-1}(\Pi^{opt})^{T}B-T^{-T}(\Pi^{opt})^TB\\
        &=0,\\
        \Psi\widehat{N}_k^{opt}\widehat{P}-\Pi^TN_kX&= 
        T^{-T}(\Pi^{opt})^TX^{opt}F\widehat{N}_k^{opt}F^{-1}T^T        -T^{-T}(\Pi^{opt})^T N_k X^{opt}T^T\\
        &= 
        T^{-T}(\Pi^{opt})^TX^{opt}FF^{-1}((\Pi^{opt})^TX^{opt})^{-1}(\Pi^{opt})^{T} N_kX^{opt}FF^{-1}T^T \\
        & \quad -T^{-T}(\Pi^{opt})^T N_k X^{opt}T^T \\
        &=0,\\
        \widehat{C}^{opt}\widehat{P}-CX &= C X^{opt}FF^{-1}T^T-CX^{opt}T^T=0,  \\
        \widehat{P}\widehat{M}_j^{opt}\widehat{P}-X^TM_jX &= TF^{-T} \widehat{M}_j^{opt} F^{-1}T^T-T(X^{opt})^TM_jX^{opt}T^T \\
        &= TF^{-T} F^T(X^{opt})^TM_j X^{opt}F F^{-1}T^T-T(X^{opt})^TM_jX^{opt}T^T \\
        &=0.
    \end{align*}
    This completes the proof.
\end{proof}

\section{Numerical Experiments} \label{sec:7}
We illustrate the performance of the proposed $\mathcal{H}_2$ based method as outlined in \Cref{alg:BQO_TSIA} on two benchmark examples by comparing it with the balanced truncation approach, including its truncated variant from \cite{FasGP25}.

For clarity, we introduce the following abbreviations for the considered algorithms:
\begin{itemize}
    \item \texttt{BT\_BQO}$(P,Q)$: Balanced truncation for BQO systems; see \cite{FasGP25}.
    \item \texttt{BT\_BQO}$(P_T,Q_T)$: Balanced truncation for BQO systems with truncated Gra\-mians; see \cite{FasGP25}.
    \item \texttt{BQO\_TSIA}: \Cref{alg:BQO_TSIA} where we solve the generalized Sylvester equations \cref{eq:alg1} by a fixed point iteration of regular Sylvester equations. 
    \item \texttt{BQO\_TSIA\_GLGMRES}: \Cref{alg:BQO_TSIA} where we solve the generalized Sylvester equations via a global variant of GMRES \cite{jbilou_global_1999}. 
\end{itemize}

The balancing-based methods \texttt{BT\_BQO}$(P,Q)$ and \texttt{BT\_BQO}$(P_T,Q_T)$ rely on the computation of Gramians via generalized Lyapunov equations or sequences of standard Lyapunov equations, followed by a Petrov--Galerkin projection. Hence, their dominant computational cost stems from solving these Lyapunov equations, which becomes prohibitive for large-scale systems. We use the choice $Q=Q^S$ from \cite{FasGP25}.

In contrast, the proposed $\mathcal{H}_2$ inspired methods require the solution of generalized Sylvester equations or sequences thereof, which can be computationally more efficient when implemented appropriately. We propose two versions of the proposed methods that differ in how these equations are solved. The first version, \texttt{BQO\_TSIA}, corresponds to \Cref{alg:BQO_TSIA}, where the generalized Sylvester equations~\eqref{eq:alg1} are solved via fixed-point iterations following \cref{thm:SylvEq}.  
The second version, denoted by \texttt{BQO\_TSIA\_GLGMRES} solves these generalized Sylvester equations using a global GMRES method \cite{jbilou_global_1999}. The solver is preconditioned by a small number of shifted solves involving the Sylvester operators $AX + X\widehat{A}^T$ and $A^T Y + Y\widehat{A}$. In contrast to \texttt{BQO\_TSIA}, this approach avoids repeated fixed-point iterations and instead treats the Sylvester equations in a single Krylov-based framework.

As with the linear, bilinear, and LQO versions of similar $\sH_2$ iterative methods, one needs to choose a stopping criterion.  One may monitor the relative change in the squared $\mathcal{H}_2$ error,
\[
\eta^{(j)}=\frac{\norm{\mathcal{S}-\widehat{\mathcal{S}}^{(j)}}_{\sH_2}^2}{\norm{\mathcal{S}}_{\sH_2}^2}=\frac{\norm{\mathcal{S}}_{\sH_2}^2 + \norm{\widehat{\mathcal{S}}^{(j)}}_{\sH_2}^2-2\skp{\mathcal{S}}{\widehat{\mathcal{S}}^{(j)}}_{\sH_2}^2}{\norm{\mathcal{S}}_{\sH_2}^2},
\]
where $\widehat{\mathcal{S}}^{(j)}$ denotes the reduced system at iteration $j.$ However, evaluating $\eta^{(j)}$ requires computing $\norm{\mathcal{S}}_{\mathcal{H}_2}$, which is expensive for large-scale systems. To avoid this, we follow \cite{reiter_2024} and instead monitor the "tail" quantity
\[
\tau^{(j)} = \norm{\widehat{\mathcal{S}}^{(j)}}_{\mathcal{H}_2}^2 - 2\skp{\mathcal{S}}{\widehat{\mathcal{S}}^{(j)}}_{\mathcal{H}_2}.
\]
It was shown in \cite{reiter_2024} that the relative change in $\tau^{(j)}$ reflects the behavior of the relative change in $\eta^{(j)}$. Therefore, we employ the stopping criterion
$\frac{\abs{\tau^{(j+1)} - \tau^{(j)}}}{\abs{\tau^{(1)}}} < 10^{-6}$.

Also, we use an initialization technique for the $\mathcal{H}_2$ based methods inspired by \cite{reiter_2024} where the initial $\widehat{A}$ is given as a diagonal matrix with logarithmically spaced entries and the other matrices are initialized as leading columns or rows of the identity matrix. For more details, see \cite{reiter_2024}. Since our main focus here was to establish the $\sH_2$ theory for BQO systems, we do not necessarily test different initialization techniques, e.g., random initialization \cite{benner_birka_2012} or mirror images of actual eigenvalues \cite{gugercin_2008}. This simple choice has worked well for our numerical examples. Nevertheless, we note that different initializations could change the speed of convergence and could also lead to converging to a different, maybe better local minimum.

Furthermore, we rescale the input $u$ to help ensure the existence of the Gramians $P$ and $Q$; see \cite{damm_direct_2008,FasGP25}.  
This rescaling $u\to\frac{1}{\gamma}u$, $0<\gamma<1$, leads to a smaller norm of the terms including $N_k$, so the associated fixed point iterations are more likely to converge. For more details, see \cite[Section 7]{FasGP25}. This also speeds up the fixed-point iterations \cref{eq:eqXitilde} for solving \cref{eq:alg1}.

All examples and algorithms are implemented in \textsc{Matlab}. For the solution of Sylvester equations as well as standard and generalized Lyapunov equations, we employ solvers from the \textsc{MESS} library. More precisely, we use the functions \texttt{mess\_lyap} and \texttt{mess\_sylvester\_sparse\_dense} (slightly modified) with default settings, as well as the function \texttt{mess\_lyapunov\_bilinear} with the following options:
\begin{itemize}
\item residual norm tolerance 
\texttt{opts.blyap.res\_tol}$=10^{-8}$, 
\item relative difference tolerance 
\texttt{opts.blyap.rel\_diff\_tol}$=10^{-7}$,  
\item maximum iterations 
\texttt{opts.blyap.maxiter}$=50$
\end{itemize}(see \cite{Mess}).
The code for the numerical examples is available on ZENODO (\url{https://doi.org/10.5281/zenodo.21477316}). 
Computations were run on an Intel(R) Core(TM) i7-1255U CPU @ 1.70 GHz with 16 GB RAM using MATLAB 2025b.

\subsection{Nonlinear RC example}

As a first example, we consider a nonlinear RC circuit from \cite{bai_projection_2006,benner_lyapunov_2011}, where the nonlinear dynamics are approximated via Carleman bilinearization \cite{rugh_nonlinear_1981}. This transformation significantly increases the system dimension, making model order reduction essential.
The resulting  BQO system has dimension $n  =40,200$ after bilinearization. We define the output as
\[
y = [1,0,\dots,0]x + \frac{1}{200^2} x^T
\begin{bmatrix}
I_{200} & 0 \\
0 & 0
\end{bmatrix} x,
\]
thereby extending the linear output from \cite{bai_projection_2006} by a quadratic term. This term corresponds to the root mean squared error (RMSE) of the original states; see, e.g., \cite{ReiW24}.
To ensure the existence of the Gramians required by the balancing-based methods, we introduce a scaling parameter $\gamma = 0.1$, according to \cite{FasGP25}. 

We test our methods on this SISO example, where we use the $u(t)=e^{-t}$ for simulation of the system and monitor the output errors between the original output $y(t)$ and the corresponding reduced order outputs $\widehat{y}(t)$. 

First, we take a look at the relative $\sH_2$ norm error between the original full system and the reduced systems. For that, we study the $\sH_2$ error from \Cref{cor:H2err} in the relative formulation as $\frac{\norm{\mathcal{S}-\widehat{\mathcal{S}}}_{\sH_2}}{\norm{\mathcal{S}}_{\sH_2}}$.
In the left-pane in \Cref{fig:1}, we plot the decay of the relative $\sH_2$ error over different reduced orders for the studied methods. We see that the balancing methods yield reduced systems with a slightly worse $\sH_2$ norm error than the systems produced by the $\sH_2$ inspired methods. Here, \texttt{BQO\_TSIA} shows convergence issue for order $8$, but is reliable for other orders.

Next, we evaluate the outputs at $N_t=10^3$ equidistant points $t_i, i=1,\dots, N_t$ in a time interval $[0,2]$. Defining the discrete-time output at these points by $y_i\approx y(t_i)$ and $\widehat{y}_i \approx \widehat{y}(t_i)$, respectively, allows us to look at the maximal output error in the Euclidean norm. The relative expression $\frac{\max_i\norm{y_i-\widehat{y}_i}}{\max_i\norm{y_i}}$ is displayed in the right-pane in \Cref{fig:1} for ascending reduced orders. We observe that the balancing algorithms \texttt{BT\_BQO}$(P,Q)$ and \texttt{BT\_BQO}$(P_T,Q_T)$ show identical results 
and provide a better output error for smaller orders from $2$ to $3$. But for higher orders, the $\sH_2$ inspired methods produce a slightly smaller output error. Comparing these $\sH_2$ inspired methods, we see small differences but the same tendency which can be explained by converging to only local minima in some cases.

\begin{figure}[h]
\centering
    \includegraphics[width=1.0\textwidth]{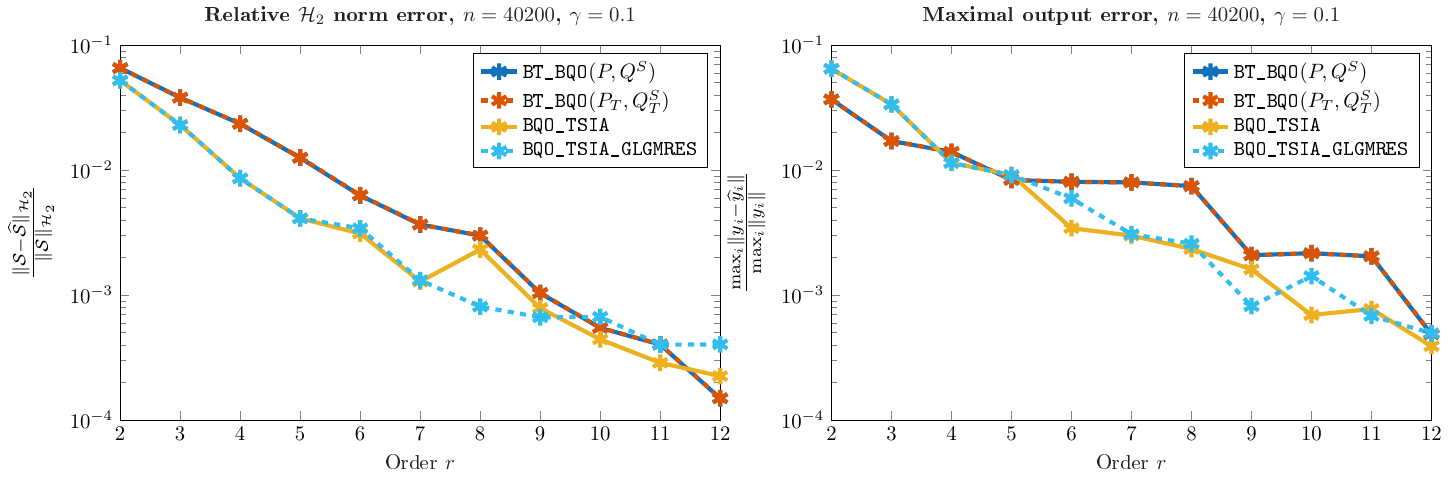}
    \caption{RC example, Left plot: Relative $\sH_2$ norm errors. Right plot: Maximal output errors.}
    \label{fig:1}
\end{figure}

Regarding the timings in \Cref{tab:1}, \texttt{BT\_BQO}$(P,Q)$ exhibits the highest computational cost, as it requires the computation of the full Gramians $P$ and $Q$. In contrast, \texttt{BT\_BQO}$(P_T,Q_T)$ is faster since it relies on approximate Gramians. Balancing-based methods have  constant run times across different reduced orders, whereas the computational cost of the $\mathcal{H}_2$ inspired methods depends on the chosen order.
Among all methods, \texttt{BQO\_TSIA\_GLGMRES} is the fastest in this example, while still delivering accurate approximations, as demonstrated in \Cref{fig:1}. For the specific order $r=10$ the method \texttt{BQO\_TSIA\_GLGMRES} takes much longer compared to the other reduced orders. Here, one can speed up the method if we initialize the reduced system by extending a converged system of smaller reduced order. However, to keep the timings and methods uniform, we retain the original formulation here. 

\begin{table}[htbp]
\footnotesize
\caption{RC example, Timings in sec}\label{tab:1}
\begin{center}
  \begin{tabular}{|c|c|c|c|c|c|c|} \hline
    Reduced Order $r$ & 2 & 4 & 6 & 8 & 10 & 12 \\
    \hline
    \texttt{BT\_BQO}$(P,Q)$ &150.91&150.91&150.91&150.91&150.91&150.91 \\ 
    \texttt{BT\_BQO}$(P_T,Q_T)$ & 21.48 & 21.48 & 21.48 & 21.48 & 21.48 & 21.48 \\
    \texttt{BQO\_TSIA}  & 11.68 & 20.34 & 87.13 & 60.62 & 60.26 & 57.49 \\
    \texttt{BQO\_TSIA\_GLGMRES} & 0.52 & 1.01 & 10.73 & 4.07 & 35.99 & 4.67 \\ 
    \hline  
  \end{tabular}
\end{center}
\end{table}

\subsection{Heat equation}
The next example is a MIMO example from \cite{benner_lyapunov_2011} which describes a bilinear controlled heat transfer system with mixed Dirichlet and Robin boundary conditions:
\begin{align*}
    x_t= &\Delta x \text{ on } \Omega :=[0,1]^2,\\
    n\cdot\nabla x = &u_1(x-1) \text{ on } \Gamma_1:=\{0\}\times [0,1),\\
    n\cdot\nabla x = &u_2(x-1) \text{ on } \Gamma_2:=(0,1]\times\{0\},\\
    x= &0 \text{ on } \Gamma_3:=\{1\}\times (0,1] \text{ and } \Gamma_4:=[0,1)\times \{1\}.
\end{align*}
These equations can be described by bilinear dynamics if we use finite differences for a discretization on an equidistant $k\times k$ grid with $k=50$. We set the output for this system as 
$
    y_1(t)=Cx(t), \,  y_2(t)=x(t)^TMx(t),
$
where $C=\frac{1}{k^2}[1,\dots,1]$, $M=\frac{1}{k^2} I_{k^2}$. Here, $y_1(t)$ shows the average temperature as suggested in \cite{benner_lyapunov_2011} whereas $y_2(t)$ represents the root mean squared error as in the RC example. We apply the decaying inputs $u_j(t)=\cos(j \pi t)\, e^{-t}, j=1,2$ to the system and study the dynamics over the time interval $t\in [0,5]$ with zero initial condition.

\begin{figure}[h]
\centering
    \includegraphics[width=1.0\textwidth]{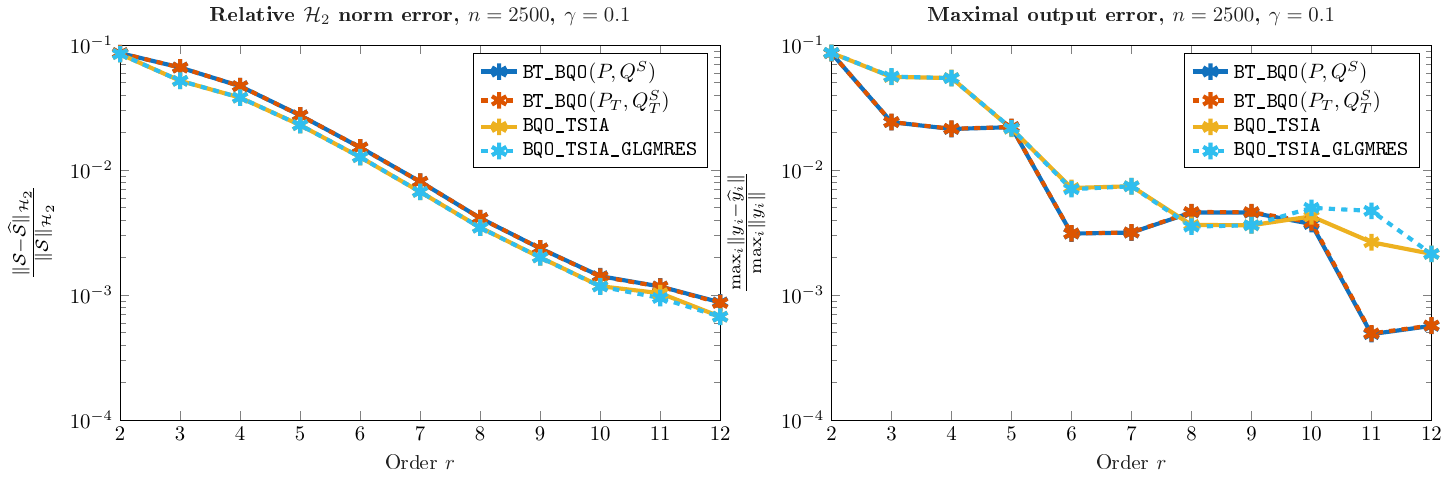}
    \caption{Heat equation example, Left plot: Relative $\sH_2$ norm errors. Right plot: Maximal output errors. }
    \label{fig:3}
\end{figure}
The left-pane in \Cref{fig:3} shows the relative $\mathcal{H}_2$ errors of the reduced systems. We see that for all reduced orders the $\sH_2$ inspired methods provide reduced systems with smaller $\sH_2$ errors than the balancing methods. 

In the right-pane plot in \Cref{fig:3}, the maximal output error, which is extended consistently to the MIMO setting, is displayed for ascending reduced orders.  
The output errors of the reduced systems produced by the balancing  methods oscillate rather than decrease monotonically, whereas the  proposed $\mathcal{H}_2$-based methods exhibit a nearly monotone  decay in the output error as the reduced order increases.

In \Cref{tab:2}, the timings for the different methods and orders are displayed. The balancing method \texttt{BT\_BQO}$(P,Q)$ is the most time-consuming whereas \texttt{BT\_BQO}$(P_T,Q_T)$ is slightly slower than the original \texttt{BQO\_TSIA} here for smaller orders. As with the previous example, \texttt{BQO\_TSIA\_GLGMRES} is the fastest, except for the $r=10$ case.

\begin{table}[htbp]
\footnotesize
\caption{Heat equation example, Timings in sec}\label{tab:2}
\begin{center}
  \begin{tabular}{|c|c|c|c|c|c|c|} \hline
    Reduced Order $r$ & 2& 4 & 6 & 8 & 10 & 12 \\
    \hline
    \texttt{BT\_BQO}$(P,Q)$ & 16.09 & 16.09 & 16.09& 16.09& 16.09& 16.09 \\
    \texttt{BT\_BQO}$(P_T,Q_T)$ & 2.02& 2.02& 2.02& 2.02& 2.02& 2.02 \\
    \texttt{BQO\_TSIA}  & 0.9 & 0.77& 1.76 & 1.4 & 2.65 & 2.29 \\  
    \texttt{BQO\_TSIA\_GLGMRES} & 0.53 & 0.4 & 0.81 & 0.7 & 4.5 & 1.27\\
    \hline
  \end{tabular}
\end{center}
\end{table}

\section{Conclusion}

In this paper, we developed a complete $\mathcal{H}_2$ framework for 
model reduction of BQO systems, which generalize both bilinear and 
LQO systems. Building on the Gramians and associated Lyapunov and 
Sylvester equations, we introduced an 
$\mathcal{H}_2$ inner product and norm for BQO systems. We then 
established output bounds  in terms of the 
$\mathcal{H}_2$ norm, motivating $\mathcal{H}_2$ optimal model 
reduction as the problem of minimizing the $\mathcal{H}_2$ system 
error. We derived the first-order optimality conditions for this problem and
proposed an algorithm that, upon convergence, produces a reduced 
BQO system satisfying these conditions. The proposed framework 
recovers the existing $\mathcal{H}_2$ theories for bilinear and LQO 
systems as special cases. Numerical experiments on two examples 
demonstrated that the proposed algorithm competes favorably with 
balanced truncation for BQO systems, and that efficient 
implementation of the arising Sylvester equations yields a method 
that is often faster while maintaining high approximation quality.

\section{Acknowledgements}
This work was supported by a fellowship of the German Academic Exchange
Service (DAAD). Thanks to Petar Mlinari\'{c} for a hint on the Lagrange function in \Cref{thm:grad} and to Tom Werner for the code on global GMRES for solving generalized Sylvester equations.

\section{Credit author statement}
 \textbf{Heike Faßbender: }Writing - Original Draft, Conceptualization, Supervision; \textbf{Serkan Gugercin: } Writing - Original Draft, Conceptualization, Methodology; \textbf{Till Peters:} Conceptualization, Methodology, Software, Investigation, Writing - Original Draft

\section{Disclosure statement} 

We wish to confirm that there are no known conflicts of interest associated with this publication and there has been no significant financial support for this work that could have influenced its outcome. Part of this work was completed during a research stay of the third author at Virginia Tech supported by a fellowship of the German Academic Exchange Service (DAAD). The work of Gugercin was supported in part by the National Science Foundation (NSF), United States under Grant No. CMMI-2130695.
We confirm that the manuscript has been read and approved by all named authors and that there are no other persons who satisfied the criteria for authorship but are not listed. We further confirm that the order of authors listed in the manuscript has been approved by all of us.
We confirm that we have given due consideration to the protection of intellectual property associated with this work and that there are no impediments to publication, including the timing of publication, with respect to intellectual property. In so doing we confirm that we have followed the regulations of our institutions concerning intellectual property.
We understand that the Corresponding Author is the sole contact for the Editorial process. He is responsible for communicating with the other authors about progress, submissions of revisions and final approval of proofs.

\bibliographystyle{siamplain}
\bibliography{bib_H2bqo}

@book{ant_beattie_gug_2020,
author = {Antoulas, A. C. and Beattie, C. A. and Güğercin, S.},
title = {Interpolatory Methods for Model Reduction},
publisher = {SIAM},
year = {2020},
doi = {10.1137/1.9781611976083},
address = {Philadelphia, PA},
edition   = {},
}

@article{xu_optimal_2011,
	title = {Optimal $\mathcal{H}_2$ {Model} {Reduction} for {Large} {Scale} {MIMO} {Systems} via {Tangential} {Interpolation}},
	volume = {8},
	copyright = {Copyright (c) 2025 COPYRIGHT: © Global Science Press},
	issn = {2617-8710},
	language = {en},
	number = {1},
	journal = {Int. J. Numer. Anal. Model.},
	author = {Xu, Y. and Zeng, T.},
	year = {2011},
	pages = {174--188},
}

@article{jbilou_global_1999,
	title = {Global {FOM} and {GMRES} algorithms for matrix equations},
	volume = {31},
	issn = {0168-9274},
	doi = {10.1016/S0168-9274(98)00094-4},
	number = {1},
	journal = {Appl. Numer. Math.},
	author = {Jbilou, K. and Messaoudi, A. and Sadok, H.},
	month = sep,
	year = {1999},
	pages = {49--63},
}

@article{flagg_multipoint_2015,
	title = {Multipoint {Volterra} {Series} {Interpolation} and $\mathcal{H}_2$ {Optimal} {Model} {Reduction} of {Bilinear} {Systems}},
	volume = {36},
	issn = {0895-4798},
	doi = {10.1137/130947830},
	number = {2},
	journal = {SIAM J. Matrix Anal. Appl.},
	author = {Flagg, Garret and Gugercin, Serkan},
	month = jan,
	year = {2015},
	pages = {549--579},
}

@book{willems70,
    author = {Willems, J. L.},
    title = {Stability Theory of Dynamical Systems},
    publisher = {Wiley, New York},
    year = {1970}
}

@unpublished{Gos24, 
	title= {Extending balanced truncation for model reduction of bilinear systems with quadratic outputs},
	author = {Ion Victor Gosea},
	note= {Talk at 24th GAMM Workshop on Applied and Numerical Linear Algebra 2024, September 23 - 24, 2024, G\"ottingen, Germany},
}

@misc{FasGP25,
	title = {Bilinear {Quadratic} {Output} {Systems} and {Balanced} {Truncation}},
	doi = {10.48550/arXiv.2507.03684},
	publisher = {arXiv},
	author = {Faßbender, Heike and Gugercin, Serkan and Peters, Till},
	month = jul,
	year = {2025},
}

@book{mohler,
  title={{Nonlinear systems (vol. 2): applications to bilinear control}},
  author={Mohler, R.R.},
  year={1991},
  publisher={Prentice-Hall, Inc. Upper Saddle River, NJ, USA}
}

@article{hart13,
        author = {Hartmann, C. and Sch\"afer-Bung, C and Th\"ons-Zueva, A.},
        title = {Balanced averaging of bilinear systems with applications 
        to stochastic control}, 
        journal = {SIAM J. Control Optim.}, 
        volume = {51},
        pages = {2356--2378}, 
        year = {2013}}

@article{PulA19,
  title={Balanced truncation for model order reduction of linear dynamical systems with quadratic outputs},
  author={Pulch, Roland and Narayan, Akil},
  journal={SIAM J. Sci. Comput.},
  volume={41},
  number={4},
  pages={A2270--A2295},
  year={2019},
  publisher={SIAM},
  doi={10.1137/17M1148797}
}

@article{Pul23,
  title={Energy-based model order reduction for linear stochastic {G}alerkin systems of second order},
  author={Pulch, Roland},
  journal={PAMM},
  volume={23},
  number={3},
  pages={e202300038},
  year={2023},
  publisher={Wiley Online Library},
  doi={10.1002/pamm.20230003833}
}

@article{YueM13,
  title={Accelerating optimization of parametric linear systems by model order reduction},
  author={Yue, Yao and Meerbergen, Karl},
  journal={SIAM J. Optim.},
  volume={23},
  number={2},
  pages={1344--1370},
  year={2013},
  publisher={SIAM},
  doi={10.1137/120869171}
}

@article{ReiW24,
  title={Interpolatory model reduction of dynamical systems with root mean squared error},
  author={Reiter, Sean and Werner, Steffen W. R.},
  journal={IFAC-PapersOnLine},
  volume={59},
  number={1},
  pages={385--390},
  year={2025},
  publisher={Elsevier},
  doi={10.1016/j.ifacol.2025.03.066}
}

@article{VanVNLM12,
  title={Model reduction for dynamical systems with quadratic output},
  author={Van Beeumen, Roel and Van Nimmen, Katrien and Lombaert, Geert and Meerbergen, Karl},
  journal={Internat. J. Numer. Methods Engrg.},
  volume={91},
  number={3},
  pages={229--248},
  year={2012},
  publisher={Wiley Online Library},
  doi={10.1002/nme.4255}
}

@article{MehU23,
  title={Control of port-{H}amiltonian differential-algebraic systems and applications},
  author={Mehrmann, Volker and Unger, Benjamin},
  journal={Acta Numerica},
  volume={32},
  pages={395--515},
  year={2023},
  publisher={Cambridge University Press},
  doi={10.1017/S0962492922000083}
}

@inproceedings{Van06,
  title={Port-{H}amiltonian systems: an introductory survey},
  author={van der Schaft, Arjan},
  booktitle={International congress of mathematicians},
  pages={1339--1365},
  year={2006},
  organization={European Mathematical Society Publishing House (EMS Ph)},
  doi={10.4171/022-3/65}
}

@phdthesis{mlinaric_structure-preserving,
    school = {OVGU Magdeburg, Germany},
	title = {Structure-{Preserving} {Model} {Order} {Reduction} for {Network} {Systems}},
	author = {Mlinaric, Petar},
    year = {2020},
    doi = {10.25673/33570},
}

@article{yan_approximate_1999,
	title = {An {Approximate} {Approach} to {Optimal} {Model} {Reduction}},
	volume = {44},
	language = {en},
	number = {7},
	journal = {IEEE Trans. Automat. Control},
	author = {Yan, Wei-Yong and Lam, James},
	year = {1999},
}

@article{benner_h2QB_2018,
	title = {$\mathcal{H}_2$-{Quasi}-{Optimal} {Model} {Order} {Reduction} for {Quadratic}-{Bilinear} {Control} {Systems}},
	volume = {39},
	issn = {0895-4798, 1095-7162},
	doi = {10.1137/16M1098280},
	language = {en},
	number = {2},
	journal = {SIAM J. Matrix Anal. Appl.},
	author = {Benner, Peter and Goyal, Pawan and Gugercin, Serkan},
	month = jan,
	year = {2018},
	pages = {983--1032},
}

@book{rugh_nonlinear_1981,
	title = {Nonlinear {System} {Theory}: {The} {Volterra}/{Wiener} {Approach}},
	isbn = {978-0-8018-2549-1},
	shorttitle = {Nonlinear {System} {Theory}},
	language = {en},
	publisher = {Johns Hopkins University Press},
	author = {Rugh, Wilson J.},
	year = {1981},
}

@article{bai_projection_2006,
	series = {Special {Issue} on {Order} {Reduction} of {Large}-{Scale} {Systems}},
	title = {A projection method for model reduction of bilinear dynamical systems},
	volume = {415},
	issn = {0024-3795},
	doi = {10.1016/j.laa.2005.04.032},
	number = {2},
	journal = {Linear Algebra Appl.},
	author = {Bai, Zhaojun and Skoogh, Daniel},
	month = jun,
	year = {2006},
	pages = {406--425},
}

@article{damm_direct_2008,
	title = {Direct methods and {ADI}-preconditioned {Krylov} subspace methods for generalized {Lyapunov} equations},
	volume = {15},
	issn = {1099-1506},
	doi = {10.1002/nla.603},
	language = {en},
	number = {9},
	journal = {Numer. Linear Algebra Appl.},
	author = {Damm, T.},
	year = {2008},
	pages = {853--871},
}

@Misc{Mess,
  author =       {Saak, J. and K\"{o}hler, M. and Benner, P.},
  title =        {{M-M.E.S.S.}-3.0 -- The Matrix Equations Sparse Solvers
                  library},
  month =        aug,
  year =         2023,
  doi =          {10.5281/zenodo.7701424},
  key =          {MMESS}
}

@Book{adrianova_1995,
 Author = {Adrianova, L. Ya.},
 Title = {Introduction to linear systems of differential equations. {Transl}. from the {Russian} by {Peter} {Zhevandrov}},
 FSeries = {Translations of Mathematical Monographs},
 Series = {Transl. Math. Monogr.},
 ISSN = {0065-9282},
 Volume = {146},
 ISBN = {0-8218-0328-X},
 Year = {1995},
 Publisher = {American Mathematical Society},
 Language = {English},
 zbMATH = {819910},
 Zbl = {0844.34001}
}

@Book{HorJ94,
  Author = {Horn, Roger A. and Johnson, Charles R.},
  Title = {Topics in matrix analysis. 1st paperback ed. with corrections},
  Edition = {1st pbk with corr.},
  ISBN = {0-521-46713-6},
  Year = {1994},
  Publisher = {Cambridge: Cambridge University Press},
  Language = {English},
  zbMATH = {635657},
  Zbl = {0801.15001}
}

@mastersthesis{padhi_2024,
	address = {Pashan, Pune India},
    school = {Indian Institute of Science Education and Research Pune},
	type = {Master's thesis},
	title = {Model {Order} {Reduction} of {Nonlinear} {Dynamical} {Systems}},
	author = {Padhi, Reetish},
	month = may,
	year = {2024},
}

@book{GolubVanLoan4th,
author = {Golub, Gene H. and Van Loan, Charles F.},
title = {Matrix Computations - 4th Edition},
publisher = {Johns Hopkins University Press},
year = {2013},
doi = {10.1137/1.9781421407944},
address = {Philadelphia, PA},
edition   = {},
}

@inproceedings{al-baiyat_new_1993,
	title = {A new model reduction scheme for k-power bilinear systems},
	doi = {10.1109/CDC.1993.325196},
	booktitle = {Proceedings of 32nd {IEEE} {Conference} on {Decision} and {Control}},
	author = {Al-Baiyat, S.A. and Bettayeb, M.},
	month = dec,
	year = {1993},
	pages = {22--27 vol.1},
}

@article{benner_birka_2012,
	title = {Interpolation-{Based} $\mathcal{H}_2$-{Model} {Reduction} of {Bilinear} {Control} {Systems}},
	volume = {33},
	issn = {0895-4798},
	doi = {10.1137/110836742},
	number = {3},
	journal = {SIAM J. Matrix Anal. Appl.},
	author = {Benner, Peter and Breiten, Tobias},
	month = jan,
	year = {2012},
	pages = {859--885},
}

@article{gugercin_2008,
	title = {$\mathcal{H}_2$ {Model} {Reduction} for {Large}-{Scale} {Linear} {Dynamical} {Systems}},
	volume = {30},
	issn = {0895-4798},
	doi = {10.1137/060666123},
	number = {2},
	journal = {SIAM J. Matrix Anal. Appl.},
	author = {Gugercin, S. and Antoulas, A. C. and Beattie, C.},
	month = jan,
	year = {2008},
	pages = {609--638},
}

@article{redmann_bilinear_2021,
	title = {Bilinear {Systems}---{A} {New} {Link} to $\mathcal{H}_2$-norms, {Relations} to {Stochastic} {Systems}, and {Further} {Properties}},
	volume = {59},
	issn = {0363-0129, 1095-7138},
	doi = {10.1137/19M1304106},
	language = {en},
	number = {4},
	journal = {SIAM J. Control Optim.},
	author = {Redmann, Martin},
	month = jan,
	year = {2021},
	pages = {2477--2497},
}

@article{benner_lyapunov_2011,
	title = {Lyapunov {Equations}, {Energy} {Functionals}, and {Model} {Order} {Reduction} of {Bilinear} and {Stochastic} {Systems}},
	volume = {49},
	issn = {0363-0129},
	doi = {10.1137/09075041X},
	number = {2},
	journal = {SIAM J. Control Optim.},
	author = {Benner, Peter and Damm, Tobias},
	month = jan,
	year = {2011},
	pages = {686--711},	
}

@article{zhang_2002,
	title = {On \textit{{H}}2 model reduction of bilinear systems},
	volume = {38},
	issn = {0005-1098},
	doi = {10.1016/S0005-1098(01)00204-7},
	number = {2},
	journal = {Automatica},
	author = {Zhang, Liqian and Lam, James},
	month = feb,
	year = {2002},
	pages = {205--216},
}

@article{benner_lqo_2022,
	title = {Gramians, {Energy} {Functionals}, and {Balanced} {Truncation} for {Linear} {Dynamical} {Systems} {With} {Quadratic} {Outputs}},
	volume = {67},
	issn = {1558-2523},
	doi = {10.1109/TAC.2021.3086319},
	number = {2},
	journal = {IEEE Trans. Automat. Control},
	author = {Benner, Peter and Goyal, Pawan and Duff, Igor Pontes},
	month = feb,
	year = {2022},
	pages = {886--893},
}

@article{benner_balanced_2024,
	title = {Balanced truncation for quadratic-bilinear control systems},
	volume = {50},
	issn = {1572-9044},
	doi = {10.1007/s10444-024-10186-9},
	language = {en},
	number = {4},
	journal = {Adv. Comput. Math. },
	author = {Benner, Peter and Goyal, Pawan},
	month = aug,
	year = {2024},
	pages = {88},
}

@misc{reiter_2024,
	title = {$\mathcal{H}_2$ optimal model reduction of linear systems with multiple quadratic outputs},
	language = {en},
	publisher = {arXiv},
	author = {Reiter, Sean and Duff, Igor Pontes and Gosea, Ion Victor and Gugercin, Serkan},
	month = may,
	year = {2024},
	note = {arXiv:2405.05951 [cs, eess, math]},
}

@incollection{BenB17,
  author    = {Peter Benner and Tobias Breiten},
  title     = {Model Order Reduction Based on System Balancing},
  booktitle = {Model Reduction and Approximation: Theory and Algorithms},
  chapter   = {6},
  pages     = {261--295},
  year      = {2017},
  publisher = {SIAM},
  doi       = {10.1137/1.9781611974829.ch6},
}

@article{GugA04,
  title={A survey of model reduction by balanced truncation and some new results},
  author={Gugercin, Serkan and Antoulas, Athanasios C},
  journal={International Journal of Control},
  volume={77},
  number={8},
  pages={748--766},
  year={2004},
doi = {10.1080/00207170410001713448},
}

@inproceedings{Wil70,
  title={Optimum solution of model-reduction problem},
  author={Wilson, David~A.},
  booktitle={Proceedings of the Institution of Electrical Engineers},
  volume={117},
  pages={1161--1165},
  year={1970},
  organization={IET},
}
\end{document}